\documentclass[11pt]{amsart}

\usepackage[T1]{fontenc}
\usepackage[utf8]{inputenc}
\usepackage{lmodern}
\usepackage{microtype}
\usepackage{mathtools}
\usepackage{amssymb}
\usepackage{mathrsfs}
\usepackage{enumitem}
\usepackage{aliascnt}
\usepackage{hyperref}
\usepackage[nameinlink,capitalize,noabbrev]{cleveref}

\hypersetup{
  colorlinks=true,
  linkcolor=blue,
  citecolor=blue,
  urlcolor=blue,
  pdftitle={Exceptional supersphere integration and logarithmic Pizzetti kernels},
  pdfauthor={Juan Bory-Reyes, Baruch Schneider{$^{\ast}$}, Diana Schneiderova, Yifan Zhang},
  pdfsubject={Orthosymplectically invariant integration, logarithmic Pizzetti kernels, and exceptional superdimension},
  pdfkeywords={supersphere integration, exceptional superdimension, Pizzetti formula, logarithmic kernel, orthosymplectic symmetry, Fischer decomposition}
}

\numberwithin{equation}{section}

\newtheorem{theorem}{Theorem}[section]
\newaliascnt{proposition}{theorem}
\newtheorem{proposition}[proposition]{Proposition}
\aliascntresetthe{proposition}
\newaliascnt{lemma}{theorem}
\newtheorem{lemma}[lemma]{Lemma}
\aliascntresetthe{lemma}
\newaliascnt{corollary}{theorem}

\aliascntresetthe{corollary}
\newaliascnt{conjecture}{theorem}

\aliascntresetthe{conjecture}
\theoremstyle{definition}
\newaliascnt{definition}{theorem}
\newtheorem{definition}[definition]{Definition}
\aliascntresetthe{definition}
\newaliascnt{example}{theorem}

\aliascntresetthe{example}
\newaliascnt{problem}{theorem}

\aliascntresetthe{problem}
\theoremstyle{remark}
\newaliascnt{remark}{theorem}
\newtheorem{remark}[remark]{Remark}
\aliascntresetthe{remark}
\crefname{theorem}{theorem}{theorems}
\Crefname{theorem}{Theorem}{Theorems}
\crefname{proposition}{proposition}{propositions}
\Crefname{proposition}{Proposition}{Propositions}
\crefname{lemma}{lemma}{lemmas}
\Crefname{lemma}{Lemma}{Lemmas}
\crefname{corollary}{corollary}{corollaries}
\Crefname{corollary}{Corollary}{Corollaries}
\crefname{definition}{definition}{definitions}
\Crefname{definition}{Definition}{Definitions}
\crefname{problem}{problem}{problems}
\Crefname{problem}{Problem}{Problems}
\crefname{remark}{remark}{remarks}
\Crefname{remark}{Remark}{Remarks}

\newcommand{\SSph}{\mathbb{S}\mathbb{S}}

\newcommand{\Ker}{\operatorname{Ker}}

\title[Exceptional supersphere integration]{Exceptional supersphere integration and logarithmic Pizzetti kernels}

\author{Juan Bory-Reyes}
\address{ESIME-Zacatenco, Instituto Polit\'ecnico Nacional, 07738 CDMX, M\'exico}
\email{juanboryreyes@yahoo.com}

\author{Baruch Schneider{$^{\ast}$}}
\address{Department of Mathematics, University of Ostrava, 70200 Ostrava, Czech Republic}
\email{baruch.schneider@osu.cz}
\thanks{$\ast$ Corresponding author}

\author{Diana Schneiderov\'a}
\address{Department of Mathematics, University of Ostrava, 70200 Ostrava, Czech Republic}
\email{diana.schneiderova@osu.cz}

\author{Yifan Zhang}
\address{Department of Mathematics, University of Ostrava, 70200 Ostrava, Czech Republic}
\address{Department of Algebra, Charles University, 18675 Prague, Czech Republic}
\address{Department of Applied Mathematics, VSB--Technical University of Ostrava, 70800 Ostrava, Czech Republic}
\email{yifan.zhang@osu.cz}

\date{}

\subjclass[2020]{Primary 30G35; Secondary 33D05, 58A50, 43A90}
\keywords{supersphere integration, exceptional superdimension, Pizzetti formula, logarithmic kernel, orthosymplectic symmetry, Fischer decomposition}

\begin{document}

\begin{abstract}
We study orthosymplectically invariant supersphere integration at the
exceptional superdimensions $M=-2u$, where the harmonic Fischer structure
becomes nonsemisimple and the Pizzetti pairing degenerates.  For the
meromorphically continued homogeneous inverse kernels we obtain the generating
function
\[
 \mathscr G_\mu(\rho;x,y)
 =\frac{\Gamma(\mu/2)}{2\pi^{\mu/2}}
  \bigl(1+\rho\{x,y\}+\rho^2x^2y^2\bigr)^{-\mu/2}.
\]
At $\mu=-2u$, its Laurent expansion has a polynomial residue and a
logarithmic finite part.  We prove that these coefficients recover the
complete degreewise duality structure on a fixed superspace with nonzero
bosonic dimension.  In degrees $k\le u$, the residue inverts a canonical
renormalized pairing on $\mathcal P_k$.  In the collision range
$u<k\le2u$, the ordinary pairing has radical
$(x^2)^{k-u}\mathcal P_{2u-k}$; the finite part reproduces the quotient,
while the residue reproduces the radical after transport from the reflected
degree.  For $k\ge2u+1$, the finite part is the ordinary inverse kernel.  We
also establish the nondegenerate head--socle pairing on the generalized
harmonic modules.  As an application, we derive covariant right--left radial
$q$-monogenic zonal symbols and identify precise degree-one obstructions to
transferring scalar Pizzetti reproduction through a one-sided $q$-Fischer
projection.
\end{abstract}

\maketitle

\section{Introduction}

Supersphere integration is the orthosymplectically invariant integration
problem associated with flat superspace.  It is controlled generically by the
superdimension $M=m-2n$, but the negative even values $M=-2u$ are singular:
the harmonic Fischer structure becomes nonsemisimple, generalized harmonic
modules replace colliding harmonic rows, and the invariant Pizzetti pairing
loses rank.  The problem is therefore not merely to continue a scalar formula
to a singular parameter, but to recover the quotient, radical, and extension
data that replace the missing inverse pairing.

Earlier work developed the Pizzetti formula and Funk--Hecke theorem in
superspace, orthosymplectic Howe duality, normalized integration in low
exceptional degrees, generalized Fischer decompositions, and exceptional
branching laws
\cite{DeBieSommen2007Spherical,CoulembierDeBieSommen2010Invariant,
Coulembier2013Orthosymplectic,GuzmanAdan2021CK,LavickaSmid2015,
Lavicka2023Branching}.  Beyond superspace, the Pizzetti principle has been
extended to invariant integration over real, complex, and quaternionic
Stiefel manifolds, with an invariant-theoretic interpretation involving Howe
dual pairs \cite{CoulembierKieburg2015}.  These works place Pizzetti-type
functionals in a broad invariant-integration framework.  They do not,
however, provide a single inverse-kernel description of all exceptional
supersphere degrees.

The first purpose of this paper is to construct such a description.  By an
\emph{inverse Pizzetti kernel} in degree $k$ we mean the kernel of the
identity operator with respect to the degree-$k$ Pizzetti pairing, or with
respect to the explicitly stated renormalized or quotient pairing when the
ordinary form degenerates.  If $\mathsf G_{M,k}$ denotes the generic inverse
kernel, then its meromorphic continuation in a formal dimension parameter
$\mu$ has the generating series
\begin{equation}\label{eq:intro-global-kernel}
 \mathscr G_\mu(\rho;x,y)
 :=\sum_{k\ge0}\rho^k\mathsf G_{\mu,k}(x,y)
 =\frac{\Gamma(\mu/2)}{2\pi^{\mu/2}}
  \bigl(1+\rho\{x,y\}+\rho^2x^2y^2\bigr)^{-\mu/2}.
\end{equation}
At $\mu=-2u$, the pole of $\Gamma(\mu/2)$ combines with the variation of the
exponent in \eqref{eq:intro-global-kernel}; differentiating the power
$D^{-\mu/2}$ therefore forces a logarithmic term.  The logarithm is not an
auxiliary regularization choice: its homogeneous coefficients are exactly
the finite-part kernels needed on the quotient and in the high-degree range.

The residue and finite part have complementary roles in the three exceptional
regimes:
\begin{enumerate}[label=\textup{(\roman*)}]
\item for $0\le k\le u$, the ordinary Pizzetti form is identically zero and
the residue is the inverse kernel of a canonical first-order renormalization
on the full space $\mathcal P_k$;
\item for $u+1\le k\le2u$, the ordinary form has exact radical
$(x^2)^{k-u}\mathcal P_{2u-k}$, the finite part reproduces the quotient, and
the residue reproduces the radical after transport from degree $2u-k$;
\item for $k\ge2u+1$, the ordinary form is nondegenerate and the finite part
is its inverse kernel on the full homogeneous space.
\end{enumerate}
We also compute the head--socle block on every exceptional generalized
harmonic module and prove that the resulting pairing is nondegenerate.  Thus
the Laurent expansion recovers the duality structure of the associated
graded exceptional Fischer filtration.  To the best of our knowledge, these
residue and finite-part kernels have not previously been assembled into one
meromorphic framework covering all three regimes.

The geometric statements are proved on a fixed superspace of superdimension
$M=-2u$ with nonzero bosonic dimension.  This hypothesis is used for the
injectivity of multiplication by the quadratic variable in the transported
radical pairing.  The purely fermionic case has additional finite-dimensional
degeneracies and is not treated here.

As an application, we consider intrinsic right radial $q$-monogenic
projection.  The abstract $q$-vector derivative and determinant-localized
Fischer projector were constructed in
\cite{BarseghyanBoryReyesSchneiderZhang2026}; their specialization to
integral superdimension, resonance, and faithful finite superspace
realizations was established in
\cite{BoryReyesSchneiderBarseghyanZhang2026}.  We derive a closed
scalar--bivector recurrence for the projected two-point zonal symbol and
prove right $q$-monogenicity in $x$, left $q$-monogenicity in $y$, and
covariance.  Degree-one calculations show precisely why scalar Pizzetti
reproduction cannot be transferred by a naive one-sided projection.  We then
formulate the resulting finite-block tensor representatives over the
localized passive-scalar field and isolate the remaining collision-range
compatibility condition.  This application is distinct from earlier
$q$-deformed Clifford operator systems
\cite{CoulembierSommen2010Q,CoulembierSommen2011} and from coordinatewise
Clifford--Jackson calculi \cite{ZimmermannBernsteinSchneider2025}.

The paper is organized as follows.  \Cref{sec:global-supersphere} proves the
meromorphic generating formula and its logarithmic Laurent expansion.
\Cref{sec:exceptional-geometry} determines the renormalized low-degree
pairing, the generalized-harmonic head--socle block, the exact radical, and
the quotient, radical, and high-degree inverse kernels.
\Cref{sec:q-application} develops the radial $q$-monogenic application and
the degree-one obstructions.

\section{Meromorphic and logarithmic supersphere kernels}\label{sec:global-supersphere}

We work in the complex polynomial superspace
\[
 \mathcal P=\mathbb C[x_1,\ldots,x_m]\otimes\Lambda_{2n},
 \qquad M=m-2n,
\]
and write $\mathcal P_k$ for its homogeneous component of total degree $k$.
The super Dirac convention used throughout is
\begin{equation}\label{eq:sign-conventions}
 r=x^2=-R^2,\qquad
 \Delta=\partial_x^2=-\Delta_{\mathrm{std}},\qquad
 [\Delta,r]=4E+2M,
\end{equation}
where $R^2$ and $\Delta_{\mathrm{std}}$ are the quadratic form and
Laplacian in the standard superspace convention, and $E$ is the Euler
operator.  Thus multiplication by $r$ restricts to multiplication by $-1$
on the formal supersphere.  We put
\[
 \mathcal H_k=\ker(\Delta)\cap\mathcal P_k,
 \qquad \Ker_k(T)=\ker(T)\cap\mathcal P_k,
\]
and use $(a)_j=\Gamma(a+j)/\Gamma(a)$ for the rising Pochhammer symbol.
All pairings below are even supersymmetric bilinear pairings; evaluation at
zero is taken coefficientwise in the Grassmann algebra.  The simultaneous
sign change in \eqref{eq:sign-conventions} leaves the standard
$\mathfrak{sl}_2$ relations, Fischer decompositions, and generalized
harmonic kernels unchanged after replacing $R^2$ by $-r$ and
$\Delta_{\mathrm{std}}$ by $-\Delta$.

For two independent supervectors write
\begin{equation}\label{eq:global-invariants}
 r=x^2,\qquad s=\{x,y\},\qquad t=y^2,
 \qquad D(\rho;x,y)=1+\rho s+\rho^2rt.
\end{equation}
For a polynomial $P$ the dimensionally continued Pizzetti functional is
\begin{equation}\label{eq:Pizzetti}
 \mathcal I_M^{\SSph}(P)
 =\sum_{j\ge0}(-1)^j
   \frac{2\pi^{M/2}}{2^{2j}j!\,\Gamma(j+M/2)}
   (\Delta^jP)(0),
\end{equation}
where the sum is finite.  For $M\notin-2\mathbb N_0$ it is the unique
orthosymplectically invariant supersphere functional with the usual
normalization; in particular
\begin{equation}\label{eq:sphere-relation}
 \mathcal I_M^{\SSph}(rP)=-\mathcal I_M^{\SSph}(P).
\end{equation}
The sphere relation also holds at the exceptional superdimensions.  Indeed,
it is enough to take $P$ homogeneous.  If $P$ has degree $2j$, then
\[
 [\Delta^{j+1},r]P
 =\sum_{h=0}^{j}\Delta^h(4E+2M)\Delta^{j-h}P,
 \qquad E\Delta^{j-h}P=2h\Delta^{j-h}P,
\]
so the commutator in \eqref{eq:sign-conventions} gives
\begin{equation}\label{eq:sphere-relation-direct}
 \bigl(\Delta^{j+1}(rP)\bigr)(0)
 =4(j+1)\left(j+\frac M2\right)(\Delta^jP)(0).
\end{equation}
Substitution into \eqref{eq:Pizzetti}, together with the entire identity
$z/\Gamma(z+1)=1/\Gamma(z)$, yields
$\mathcal I_M^{\SSph}(rP)=-\mathcal I_M^{\SSph}(P)$ also when
$M\in-2\mathbb N_0$.  If $P$ has odd degree, both sides vanish, and the
general statement follows by homogeneous decomposition.  Only the uniqueness
assertion above is restricted to the nonexceptional case.  The generic
supersphere functional, harmonic orthogonality, and the superspace
Funk--Hecke theorem are developed in
\cite{DeBieSommen2007Spherical,CoulembierDeBieSommen2010Invariant}.
For an actual superspace, the parameter in \eqref{eq:Pizzetti} is its fixed
superdimension $M=m-2n$.  Later we use a different symbol $\mu$ for the
meromorphic continuation of the explicit invariant kernel coefficients.  We
do \emph{not} identify nearby values of $\mu$ with operators on the same fixed
coordinate superspace.  Every exceptional reproducing identity below is
proved directly at the fixed value $M=-2u$.

The following theorem gives the inverse Pizzetti kernel in a fixed
homogeneous degree.  It also fixes the normalization used in the global
generating series.

\begin{theorem}\label{thm:generic-homogeneous-kernel}
For $M\notin-2\mathbb N_0$ and $k\ge0$, define
\begin{align}
 H_{M,k}^{\mathrm{amb}}(x,y)
 &=\sum_{j=0}^{\lfloor k/2\rfloor}
 (-1)^j\frac{k!}{j!(k-2j)!}
 \frac{(M/2)_{k-j}}{(M/2)_k}
 (rt)^js^{k-2j},\label{eq:ambient-seed-coefficients}\\
 \gamma_{M,k}
 &=(-1)^k\frac{\Gamma(M/2+k)}{2k!\,\pi^{M/2}},\label{eq:ambient-gamma-closed}\\
 \mathsf G_{M,k}(x,y)&:=\gamma_{M,k}H_{M,k}^{\mathrm{amb}}(x,y).
 \label{eq:meromorphic-homogeneous-kernel}
\end{align}
Equivalently,
\begin{equation}\label{eq:ambient-gegenbauer}
 H_{M,k}^{\mathrm{amb}}(x,y)
 =(-1)^k\frac{k!}{(M/2)_k}(rt)^{k/2}
 C_k^{M/2}\!\left(-\frac{s}{2\sqrt{rt}}\right),
\end{equation}
where the square roots only abbreviate the polynomial expansion.  For every
$P\in\mathcal P_k$,
\begin{equation}\label{eq:superspace-ambient-reproduction}
 \mathcal I_{M,y}^{\SSph}\!\left(\mathsf G_{M,k}(x,y)P(y)\right)=P(x).
\end{equation}
\end{theorem}

\begin{proof}
For $M\notin-2\mathbb N_0$ the Fischer decomposition is the direct sum
\[
 \mathcal P_k=\bigoplus_{j=0}^{\lfloor k/2\rfloor}
 r^j\mathcal H_{k-2j}.
\]
By \eqref{eq:sphere-relation} and harmonic orthogonality, distinct summands
are orthogonal for the Pizzetti pairing.  Let
$K_{M,\ell}^{\mathcal H}(x,y)$ denote the standard zonal reproducing kernel
of $\mathcal H_\ell$; its Gegenbauer expression and normalization are given
in \cite[Sections~4--5]{DeBieSommen2007Spherical} and
\cite[Section~3]{DeBieSommenWutzig2016}.  We normalize the harmonic kernel by
\[
 \mathcal I_{M,y}^{\SSph}\!\left(
 K_{M,\ell}^{\mathcal H}(x,y)H(y)\right)=H(x),
 \qquad H\in\mathcal H_\ell.
\]
The inverse kernel on the row $r^j\mathcal H_\ell$, where
$\ell=k-2j$, is then $r^jt^jK_{M,\ell}^{\mathcal H}(x,y)$.  Indeed, its
pairing with $t^jH(y)$ contains the factor $t^{2j}$, which equals $1$ under
the sphere functional, and the remaining harmonic integral returns
$r^jH(x)$.  The inverse kernel on $\mathcal P_k$ is therefore the finite
orthogonal sum of these row kernels.

With the above Pizzetti normalization of the harmonic kernels, the finite
connection identity is
\begin{align}\label{eq:normalized-row-connection}
 \sum_{j=0}^{\lfloor k/2\rfloor}
 r^jt^jK_{M,k-2j}^{\mathcal H}(x,y)
 &=\gamma_{M,k}(-1)^k\frac{k!}{(M/2)_k}(rt)^{k/2}\notag\\
 &\quad\times C_k^{M/2}\!\left(-\frac{s}{2\sqrt{rt}}\right)
 =\mathsf G_{M,k}(x,y).
\end{align}
The factor $\gamma_{M,k}$ is essential: for $k=0$, the left-hand side is the
constant Pizzetti reproducing kernel
$\Gamma(M/2)/(2\pi^{M/2})$, not $1$.  Formula
\eqref{eq:normalized-row-connection} is the normalized Gegenbauer connection
formula; see \cite[Chapter~IV, Sections~4.7--4.9]{Szego1975}, together with
the harmonic Pizzetti normalization in
\cite[Sections~4--5]{DeBieSommen2007Spherical}.  Expanding
\[
 C_k^\lambda(z)=\sum_{j=0}^{\lfloor k/2\rfloor}
 (-1)^j\frac{(\lambda)_{k-j}}{j!(k-2j)!}(2z)^{k-2j}
\]
gives \eqref{eq:ambient-seed-coefficients}, and the leading coefficient is
exactly \eqref{eq:ambient-gamma-closed}.  Hence the orthogonal sum of the row
inverse kernels is $\mathsf G_{M,k}$, proving
\eqref{eq:superspace-ambient-reproduction}.
\end{proof}

The coefficients in \eqref{eq:global-coefficient-formula} below are
meromorphic functions of the dimension appearing in the generic formula.  To
avoid confusing this continuation with the fixed superdimension of a
coordinate realization, we henceforth denote the continuation parameter by
$\mu$.  We write $\mathsf G_{\mu,k}$ and
$H_{\mu,k}^{\mathrm{amb}}$ for the expressions obtained from
\eqref{eq:meromorphic-homogeneous-kernel} and
\eqref{eq:ambient-seed-coefficients}, respectively, by replacing $M$ by
$\mu$.

The degreewise kernels admit a single closed generating function.  Its
Laurent expansion at a negative even value of the continuation parameter is
the source of the polynomial residue and logarithmic finite-part kernels used
below.

\begin{theorem}\label{thm:global-log-kernel}
As a formal power series in $\rho$ with coefficients meromorphic in $\mu$,
\begin{equation}\label{eq:global-generating-kernel}
 \mathscr G_\mu(\rho;x,y)
 :=\sum_{k=0}^{\infty}\rho^k\mathsf G_{\mu,k}(x,y)
 =\frac{\Gamma(\mu/2)}{2\pi^{\mu/2}}D(\rho;x,y)^{-\mu/2}.
\end{equation}
Let $u\in\mathbb N_0$ and put $H_u=\sum_{j=1}^u j^{-1}$, with $H_0=0$.
Here $\gamma_{\!E}$ denotes Euler's constant.  Near $\mu=-2u$ one has
\begin{equation}\label{eq:global-laurent-expansion}
 \mathscr G_\mu
 =\frac{\mathscr R_u}{\mu+2u}+\mathscr L_u+O(\mu+2u),
\end{equation}
where
\begin{align}
 \mathscr R_u(\rho;x,y)
 &=\frac{(-1)^u\pi^u}{u!}D(\rho;x,y)^u,
 \label{eq:global-polynomial-residue}\\
 \mathscr L_u(\rho;x,y)
 &=\frac{(-1)^u\pi^u}{2u!}D(\rho;x,y)^u
 \left(H_u-\gamma_{\!E}-\log\pi-\log D(\rho;x,y)\right).
 \label{eq:global-log-finite-part}
\end{align}
Here $\log D$ is the formal series determined by $D(0;x,y)=1$.  In
particular, the residue has degree $2u$ in $\rho$, whereas the logarithmic
finite part has homogeneous coefficients in every degree.
\end{theorem}

\begin{proof}
Multiplying \eqref{eq:ambient-seed-coefficients} by
\eqref{eq:ambient-gamma-closed} yields
\begin{equation}\label{eq:global-coefficient-formula}
 \mathsf G_{\mu,k}(x,y)
 =\sum_{j=0}^{\lfloor k/2\rfloor}
 \frac{(-1)^{k+j}\Gamma(\mu/2+k-j)}
 {2\pi^{\mu/2}j!(k-2j)!}(rt)^js^{k-2j}.
\end{equation}
To identify the coefficient of $\rho^k(rt)^js^{k-2j}$ on the right-hand
side of \eqref{eq:global-generating-kernel}, first choose $k-j$ factors
from the binomial expansion of $D^{-\mu/2}$ and then choose $j$ copies of
$rt$ among them.  This gives
\[
 \frac{\Gamma(\mu/2)}{2\pi^{\mu/2}}
 (-1)^{k-j}\frac{(\mu/2)_{k-j}}{(k-j)!}
 \binom{k-j}{j}
 =\frac{(-1)^{k+j}\Gamma(\mu/2+k-j)}
 {2\pi^{\mu/2}j!(k-2j)!},
\]
which is exactly the coefficient in
\eqref{eq:global-coefficient-formula}.  This proves the generating identity.  Now put
$\mu=-2u+\varepsilon$ and
$\delta=\varepsilon/2$.  The standard expansions
\[
 \Gamma(-u+\delta)
 =\frac{(-1)^u}{u!}\left(\frac1\delta+H_u-\gamma_{\!E}+O(\delta)\right),
\]
\[
 \pi^{u-\delta}=\pi^u(1-\delta\log\pi+O(\delta^2)),\qquad
 D^{u-\delta}=D^u(1-\delta\log D+O(\delta^2))
\]
give \eqref{eq:global-laurent-expansion}--\eqref{eq:global-log-finite-part}.
\end{proof}

The preceding logarithm has a direct spherical interpretation.

\begin{remark}
On the formal supersphere, $r=t=-1$ and
$s=-2\langle x,y\rangle$, so $D=1-2\rho\langle x,y\rangle+\rho^2$ is the
classical Poisson denominator.  Negative even superdimension changes its
complex power into a polynomial residue and a logarithmic finite part.  The
next section shows that their homogeneous coefficients are exactly the
inverse kernels of the exceptional Fischer layers.
\end{remark}

\section{Exceptional Fischer geometry}\label{sec:exceptional-geometry}

Throughout this section, $M=-2u$ and the bosonic part is nonzero.  The
purely fermionic Fischer decomposition is finite and has a different
structure; it is not considered here; see \cite{DeBieEelbodeSommen2009}.

We begin with the elementary obstruction that necessitates a residue
normalization in low degree.

\begin{proposition}\label{prop:exceptional-pizzetti-obstruction}
Let $M=-2u$ with $u\in\mathbb N_0$.  If $0\leq k\leq u$, then
\begin{equation}\label{eq:exceptional-pizzetti-zero}
 \mathcal I_{-2u}^{\SSph}(P_kQ_k)=0
 \qquad
 \text{for all }P_k,Q_k\in\mathcal P_k.
\end{equation}
Hence the unrenormalized Pizzetti pairing is identically zero on degree $k$,
and no bidegree-$(k,k)$ kernel can reproduce every nonzero homogeneous
polynomial of degree $k$ by direct substitution in $M$.
\end{proposition}

\begin{proof}
The product $P_kQ_k$ is homogeneous of degree $2k$, so in
\eqref{eq:Pizzetti} only the term with $j=k$ can survive evaluation at the
origin.  Its coefficient contains
\[
 \frac{1}{\Gamma(k+M/2)}=\frac{1}{\Gamma(k-u)}=0
 \qquad(0\leq k\leq u).
\]
This proves \eqref{eq:exceptional-pizzetti-zero}.  A reproducing kernel would
make the pairing nonzero on every nonzero vector, which is impossible.
\end{proof}

For a homogeneous polynomial $F$ of degree $2k$, the formal dimension
continuation of the Pizzetti functional is
\begin{equation}\label{eq:pizzetti-homogeneous-coefficient}
 \mathcal I_M^{\SSph}(F)
 =(-1)^k\frac{2\pi^{M/2}}{4^k k!\,\Gamma(k+M/2)}
   (\Delta^kF)(0).
\end{equation}
At $M=-2u$ and $k\leq u$ this coefficient has a simple zero, rather than an
essential degeneracy.  It therefore defines the following canonical
fixed-realization residue functional.

\begin{definition}\label{def:pizzetti-residue}
Let the fixed superspace have superdimension $M=-2u$, and let $0\leq k\leq u$.  Write
\[
 c_k(\mu)=(-1)^k\frac{2\pi^{\mu/2}}
 {4^k k!\,\Gamma(k+\mu/2)}.
\]
On homogeneous polynomials of degree $2k$ define the degreewise residue by
\begin{equation}\label{eq:pizzetti-residue-definition}
 \mathcal R_{u,k}(F)
 :=c_k'(-2u)(\Delta^kF)(0).
\end{equation}
Thus only the scalar Pizzetti normalization is continued; the Laplacian is the
one belonging to the fixed superspace.  Explicitly,
\begin{equation}\label{eq:pizzetti-residue-explicit}
 \mathcal R_{u,k}(F)
 =(-1)^u\frac{(u-k)!}{4^k k!\,\pi^u}
   (\Delta^kF)(0).
\end{equation}
Equivalently, this is the coefficient of $\varepsilon$ in
$c_k(-2u+\varepsilon)$, but it does not assert the existence of a nearby
coordinate realization of superdimension $-2u+\varepsilon$.
\end{definition}

The residue is closely related to the normalized supersphere functional of
\cite[Section~3]{GuzmanAdan2021CK}.  On polynomials of degree at most
$2u+1$, that functional is
\begin{equation}\label{eq:exceptional-normalized-integral}
 \mathcal N_u(F)
 :=\frac1{u!}\sum_{j=0}^{u}
 \frac{(u-j)!}{j!}\left(\frac{\Delta}{4}\right)^jF(0).
\end{equation}
For a homogeneous polynomial $F$ of degree $2k$, $0\leq k\leq u$, only the
term $j=k$ survives and therefore
\begin{equation}\label{eq:residue-normalized-comparison}
 \mathcal R_{u,k}(F)
 =(-1)^u\frac{u!}{\pi^u}\,\mathcal N_u(F).
\end{equation}
The signs in \eqref{eq:exceptional-normalized-integral} reflect the
convention $\Delta=-\Delta_{\mathrm{std}}$ in
\eqref{eq:sign-conventions}.  Thus the low-degree residue is a
degree-independent scalar multiple of the previously defined normalized
integral.  The additional content below consists of the nondegeneracy of the induced
bilinear form on each full homogeneous space, its inverse kernel, and the
extension to degrees for
which \eqref{eq:exceptional-normalized-integral} is not defined.

The residue form is diagonal with respect to the ordinary Fischer rows,
and its restriction to each row can be computed explicitly.

\begin{lemma}\label{lem:exceptional-residue-rows}
Let $0\leq k\leq u$, let $\ell=k-2j$, and take
$H,H'\in\mathcal H_\ell$.  Then distinct harmonic rows are orthogonal for the
residue pairing, while on the row $r^j\mathcal H_\ell$ one has
\begin{equation}\label{eq:exceptional-residue-row-scalar}
 \mathcal R_{u,k}\!\left(r^{2j}HH'\right)
 =(-1)^u\frac{(u-\ell)!}{4^\ell\ell!\,\pi^u}
   \mathsf F_\ell(H,H'),
 \qquad
 \mathsf F_\ell(H,H')=(\Delta^\ell(HH'))(0).
\end{equation}
In particular, every Fischer row carries a nonzero scalar multiple of the
Fischer contraction.
\end{lemma}

\begin{proof}
By the comparison \eqref{eq:residue-normalized-comparison}, it is
enough to evaluate the normalized functional $\mathcal N_u$.  This
functional satisfies the sphere relation and harmonic orthogonality on its
domain \cite[Section~3]{GuzmanAdan2021CK}.  Hence distinct Fischer rows are
orthogonal, while
\[
 \mathcal N_u(r^{2j}HH')=\mathcal N_u(HH')
 =\frac{(u-\ell)!}{u!\,4^\ell\ell!}
   (\Delta^\ell(HH'))(0).
\]
Multiplication by $(-1)^uu!/\pi^u$ gives
\eqref{eq:exceptional-residue-row-scalar}.  The Fischer contraction on
$\mathcal H_\ell$ is nondegenerate because the ordinary Fischer pairing is
nondegenerate and the Fischer decomposition is direct in degree
$\ell\le u$; see \cite{DeBieSommen2007Spherical,LavickaSmid2015}.  The
scalar in \eqref{eq:exceptional-residue-row-scalar} is nonzero.
\end{proof}

The row calculation yields a nondegenerate residue form and an explicit
inverse kernel on every low-degree homogeneous space.

\begin{theorem}\label{thm:exceptional-residue-reproduction}
Let $M=-2u$, $0\leq k\leq u$, and let
$H_{-2u,k}^{\mathrm{res}}(x,y)$ denote the coefficientwise specialization
\begin{equation}\label{eq:exceptional-ambient-limit}
 H_{-2u,k}^{\mathrm{res}}(x,y)
 :=\lim_{\mu\to-2u}H_{\mu,k}^{\mathrm{amb}}(x,y).
\end{equation}
This limit exists.  The residue pairing
\begin{equation}\label{eq:exceptional-residue-pairing}
 (P,Q)\longmapsto \mathcal R_{u,k}(PQ)
\end{equation}
is nondegenerate on $\mathcal P_k$, and
\begin{equation}\label{eq:exceptional-residue-reproduction}
 \widehat\gamma_{u,k}\,
 \mathcal R_{u,k,y}\!\left(
 H_{-2u,k}^{\mathrm{res}}(x,y)P(y)
 \right)=P(x),
 \qquad
 \widehat\gamma_{u,k}
 =(-1)^u\frac{\pi^u}{k!(u-k)!}.
\end{equation}
Equivalently,
\begin{equation}\label{eq:exceptional-delta-reproduction}
 \frac{1}{4^k(k!)^2}
 \left.\Delta_y^k\!\left(
 H_{-2u,k}^{\mathrm{res}}(x,y)P(y)
 \right)\right|_{y=0}
 =P(x).
\end{equation}
Thus the zero of the ordinary supersphere functional in degrees $k\leq u$ is
removed by a canonical first residue, without passing to a quotient.
\end{theorem}

\begin{proof}
The expansion of the reciprocal gamma function gives
\eqref{eq:pizzetti-residue-explicit}.  For $k\leq u$, the ordinary Fischer
decomposition is direct because these degrees lie below the first exceptional
index $u+2$; see \cite[Theorem~1]{LavickaSmid2015}.  By
\Cref{lem:exceptional-residue-rows}, the residue form is the orthogonal sum of
nonzero multiples of the Fischer contractions on the rows
$r^j\mathcal H_{k-2j}$, and is therefore nondegenerate.

Let $K_{u,k}^{\mathrm{row}}$ be the inverse kernel obtained by summing the
inverse kernels of these orthogonal rows.  In the normalized connection
identity \eqref{eq:normalized-row-connection}, both the generic Pizzetti form
and its row inverse kernels have a simple zero/pole at $\mu=-2u$ when
$k\le u$.  Taking the coefficient of $(\mu+2u)^{-1}$ in this finite
polynomial identity gives
\[
 K_{u,k}^{\mathrm{row}}
 =\operatorname*{res}_{\mu=-2u}\gamma_{\mu,k}\,
  H_{-2u,k}^{\mathrm{res}}
 =\widehat\gamma_{u,k}H_{-2u,k}^{\mathrm{res}}.
\]
This step continues only the explicit scalar row normalizations.  The
reproducing statement itself is on the fixed exceptional superspace and
follows from the direct row calculation in
\Cref{lem:exceptional-residue-rows}; it is not obtained by varying the
coordinate superdimension.  This proves
\eqref{eq:exceptional-residue-reproduction}.  Finally, multiplying the
constants in \eqref{eq:pizzetti-residue-explicit} and
\eqref{eq:exceptional-residue-reproduction} yields
$1/(4^k(k!)^2)$, which is exactly
\eqref{eq:exceptional-delta-reproduction}.
\end{proof}

We shall use the following polarization identity to evaluate the
head--socle block of the exceptional pairing.

\begin{lemma}\label{lem:polarized-laplacian}
Let $A\in\mathcal P_{2a+h}$ and $H\in\mathcal H_h$.  Assume that
$\Delta^aA\in\mathcal H_h$.  Then
\begin{equation}\label{eq:polarized-laplacian}
 \left.\Delta^{a+h}(AH)\right|_0
 =\frac{(a+h)!}{a!h!}\,
   \mathsf F_h(\Delta^aA,H),
 \qquad
 \mathsf F_h(U,V)=(\Delta^h(UV))(0).
\end{equation}
\end{lemma}

\begin{proof}
Set
\[
 \mathcal C_{x,y}:=\tfrac12
 (\Delta_{x+y}-\Delta_x-\Delta_y),
\]
where $x$ and $y$ are independent supervectors.  Polarization of the
diagonal product gives
\[
 \left.\Delta^{a+h}(AH)\right|_0
 =\left.(\Delta_x+\Delta_y+2\mathcal C_{x,y})^{a+h}
   (A(x)H(y))\right|_{x=y=0}.
\]
Every term containing $\Delta_y$ vanishes because $H$ is harmonic.  Homogeneity forces the surviving term to contain exactly $a$
copies of $\Delta_x$ and $h$ copies of $2\mathcal C_{x,y}$; all other terms
have positive residual degree in at least one variable.  Its multinomial
coefficient is $(a+h)!/(a!h!)$.  Since $\Delta^aA$ and $H$ are harmonic of
degree $h$, polarization of the Fischer contraction gives
\[
 (2\mathcal C_{x,y})^h
 \bigl((\Delta_x^aA)(x)H(y)\bigr)\big|_{x=y=0}
 =\mathsf F_h(\Delta^aA,H).
\]
Combining these identities proves \eqref{eq:polarized-laplacian}.
\end{proof}

The next lemma supplies the missing pairing between the head and socle of
an exceptional generalized harmonic module.

\begin{lemma}\label{lem:confluent-exceptional-pairing}
Let $M=-2u$, let
\[
 \ell\in I_M=\{u+2,\ldots,2u+2\},\qquad
 h=2u+2-\ell,\qquad a=u-h+1=\ell-u-1,
\]
and let $\widetilde{\mathcal H}_\ell=\Ker_\ell(\Delta r\Delta)$ be the
generalized harmonic module of \cite{LavickaSmid2015}.  The ordinary Pizzetti
form
\[
 \mathsf B_\ell(P,Q)=\mathcal I_{-2u}^{\SSph}(PQ)
\]
is nondegenerate on $\widetilde{\mathcal H}_\ell$.  More precisely,
\[
 \mathcal H_\ell^0=r^a\mathcal H_h
 \subset\mathcal H_\ell\subset\widetilde{\mathcal H}_\ell
\]
is the composition series of \cite[Theorem~2]{LavickaSmid2015}; the restriction of
$\mathsf B_\ell$ to $\mathcal H_\ell$ has radical
$\mathcal H_\ell^0$, and the induced pairing between
$\widetilde{\mathcal H}_\ell/\mathcal H_\ell$ and
$\mathcal H_\ell^0$ is nondegenerate.

More explicitly, put
\[
 \mathsf F_h(H,H')=(\Delta^h(HH'))(0),\qquad H,H'\in\mathcal H_h.
\]
For every $H\in\mathcal H_h$ one may choose a lift
$Q_H\in\widetilde{\mathcal H}_\ell$ satisfying
\begin{equation}\label{eq:exceptional-lift-equation}
 \Delta Q_H=r^{a-1}H,
\end{equation}
equivalently $r\Delta Q_H=r^aH$.  Its class modulo $\mathcal H_\ell$ depends
only on $H$, and
\begin{equation}\label{eq:exceptional-head-socle-pairing}
 \mathsf B_\ell(Q_H,r^aH')
 =\rho_{u,h}\,\mathsf F_h(H,H'),
 \qquad
 \rho_{u,h}=(-1)^u
 \frac{(u-h)!}{2a\,4^h h!\,\pi^u}\neq0.
\end{equation}
\end{lemma}

\begin{proof}
The Pizzetti form is an even supersymmetric
$\mathfrak{osp}(m|2n)$-invariant pairing.  Its restriction to
$\mathcal H_\ell$ therefore has an invariant radical.  The maximal proper
submodule is $\mathcal H_\ell^0$ by
\cite[Theorem~2]{LavickaSmid2015}.  This submodule is contained in the
radical.  Indeed, write $b_H=r^aH$ with $H\in\mathcal H_h$.  For
$Q\in\mathcal H_\ell$, the sphere relation gives
$\mathsf B_\ell(b_H,Q)=(-1)^a\mathcal I_{-2u}^{\SSph}(HQ)$, which is zero
by harmonic-degree orthogonality because $h\ne\ell$; see
\cite[Section~4]{DeBieSommen2007Spherical}.  If both arguments
belong to $\mathcal H_\ell^0$, the same reduction gives the harmonic
degree-$h$ normalization; it vanishes because $h\leq u$ and the reciprocal
gamma factor is zero.  On the other hand, the harmonic-degree-$\ell$
normalization is nonzero because $\Gamma(\ell-u)$ is finite.  The standard harmonic Fischer--Pizzetti
comparison therefore shows that the form induced on
$\mathcal H_\ell/\mathcal H_\ell^0$ is not identically zero.  Since this
quotient is irreducible and $\mathcal H_\ell^0$ is the maximal proper
submodule, the invariant radical is exactly $\mathcal H_\ell^0$; see
\cite[Theorem~2]{LavickaSmid2015} and
\cite[Sections~4--6]{Coulembier2013Orthosymplectic}.  Hence
\[
 \operatorname{Rad}(\mathsf B_\ell|_{\mathcal H_\ell})
 =\mathcal H_\ell^0.
\]

We now compute the missing head--socle block directly at the exceptional
superdimension.  Since $m>0$, the super Laplacian is surjective from
$\mathcal P_\ell$ onto $\mathcal P_{\ell-2}$; see the proof of
\cite[Lemma~5]{LavickaSmid2015}.  The sign convention
\eqref{eq:sign-conventions} changes neither this statement nor the kernels
of the generalized harmonic operators.  Hence, for
$H\in\mathcal H_h$, choose $Q_H\in\mathcal P_\ell$ such that
\[
 \Delta Q_H=r^{a-1}H.
\]
Put $b_H=r^aH$.  The radial transition formula
\cite[Lemma~2]{LavickaSmid2015} gives, at $M=-2u$,
\begin{equation}\label{eq:exceptional-collision-transition}
 \Delta b_H=2a(2h-2u+2a-2)r^{a-1}H=0.
\end{equation}
Moreover, iterating the same formula yields
\begin{equation}\label{eq:exceptional-lift-iterate}
 \Delta^aQ_H
 =\Delta^{a-1}(r^{a-1}H)
 =(-4)^{a-1}((a-1)!)^2H.
\end{equation}

It remains to evaluate the Pizzetti form.  Since $r=-1$ under the
supersphere functional and $a+h=u+1$, we have
\[
 \mathsf B_\ell(Q_H,b_{H'})=(-1)^a
 \mathcal I_{-2u}^{\SSph}(Q_HH').
\]
Applying \Cref{lem:polarized-laplacian} with $A=Q_H$ gives
\begin{equation}\label{eq:exceptional-product-laplacian}
 \left.\Delta^{u+1}(Q_HH')\right|_0
 =\frac{(u+1)!}{a!h!}
 \mathsf F_h(\Delta^aQ_H,H').
\end{equation}
Using the homogeneous Pizzetti coefficient with $k=u+1$, followed by
\eqref{eq:exceptional-lift-iterate}, gives
\begin{align*}
 \mathsf B_\ell(Q_H,b_{H'})
 &=(-1)^{a+u+1}
   \frac{2\pi^{-u}}{4^{u+1}(u+1)!}
   \frac{(u+1)!}{a!h!}
   (-4)^{a-1}((a-1)!)^2
   \mathsf F_h(H,H')\\
 &=(-1)^u\frac{(a-1)!}{2a\,4^hh!\,\pi^u}
   \mathsf F_h(H,H')
 =\rho_{u,h}\mathsf F_h(H,H'),
\end{align*}
since $a-1=u-h$.  This proves
\eqref{eq:exceptional-head-socle-pairing}.

Equation \eqref{eq:exceptional-lift-equation} implies
\[
 \Delta r\Delta Q_H=\Delta(r^aH)=0,
\]
so $Q_H\in\widetilde{\mathcal H}_\ell$, while
$r\Delta Q_H=r^aH$ shows that its quotient class maps to the socle under the
surjection
$r\Delta:\widetilde{\mathcal H}_\ell/\mathcal H_\ell
\to\mathcal H_\ell^0$ of \cite[Theorem~2]{LavickaSmid2015}.  If a different lift is chosen,
the difference lies in
$\mathcal H_\ell$.  The already proved identity
$\operatorname{Rad}(\mathsf B_\ell|_{\mathcal H_\ell})=
\mathcal H_\ell^0$, together with supersymmetry of the Pizzetti form, gives
\[
 \mathsf B_\ell(\mathcal H_\ell,\mathcal H_\ell^0)=0.
\]
Equivalently, this follows by harmonic-degree orthogonality after writing
$\mathcal H_\ell^0=r^a\mathcal H_h$.  Hence the cross-pairing depends only on
the quotient class of the lift and is well-defined.  The form $\mathsf F_h$ is nondegenerate and
$\rho_{u,h}\neq0$, so the head pairs nondegenerately with the socle.  If a
vector lies in the radical of $\mathsf B_\ell$ on
$\widetilde{\mathcal H}_\ell$, pairing with the socle first forces its
class in $\widetilde{\mathcal H}_\ell/\mathcal H_\ell$ to vanish.
The vector therefore lies in $\mathcal H_\ell$; nondegeneracy on
$\mathcal H_\ell/\mathcal H_\ell^0$ then forces it into the socle.
Finally, the nondegenerate head--socle pairing forces the socle component to be
zero.  Thus $\mathsf B_\ell$ is nondegenerate on the full generalized
harmonic module.
\end{proof}

We can now determine the radical of the ordinary Pizzetti form in every
homogeneous degree above the low-degree residue range.

\begin{theorem}\label{thm:exceptional-full-radical}
Let $M=-2u$ and $k\geq u+1$.  Put
\begin{equation}\label{eq:exceptional-radical-space}
 \mathfrak R_{u,k}:=
 \begin{cases}
  r^{k-u}\mathcal P_{2u-k},&u+1\leq k\leq2u,\\
  0,&k\geq2u+1.
 \end{cases}
\end{equation}
Then
\begin{equation}\label{eq:exceptional-radical-identity}
 \operatorname{Rad}\bigl(
   (P,Q)\mapsto\mathcal I_{-2u}^{\SSph}(PQ)
 \bigr)=\mathfrak R_{u,k}
 \qquad\text{on }\mathcal P_k.
\end{equation}
Consequently, the ordinary Pizzetti form descends nondegenerately to
\begin{equation}\label{eq:exceptional-full-quotient}
 \overline{\mathcal P}_{u,k}:=
 \mathcal P_k/\mathfrak R_{u,k}.
\end{equation}
In particular, it is already nondegenerate on the full homogeneous space for
every $k\geq2u+1$.
\end{theorem}

\begin{proof}
Put
\[
 N_k=\{k-2j:0\le j\le\lfloor k/2\rfloor\},\qquad
 I_M=\{u+2,\ldots,2u+2\},
\]
\[
 \widetilde J_k=N_k\cap I_M,\qquad
 J_k^0=\{2u+2-\ell:\ell\in\widetilde J_k\},\qquad
 J_k=N_k\setminus(\widetilde J_k\cup J_k^0).
\]
With this notation, \cite[Theorem~1 and Corollary~1]{LavickaSmid2015}
gives the generalized Fischer decomposition
\begin{equation}\label{eq:exceptional-generalized-fischer}
 \mathcal P_k=
 \bigoplus_{\ell\in\widetilde J_k}
 r^{(k-\ell)/2}\widetilde{\mathcal H}_\ell
 \ \oplus\!
 \bigoplus_{\ell\in J_k}
 r^{(k-\ell)/2}\mathcal H_\ell.
\end{equation}
Consider the even operator
\begin{equation}\label{eq:casimir-definition}
 \Omega=E(E+M-2)-r\Delta.
\end{equation}
In the present conventions, $\Omega$ is the spherical
Laplace--Beltrami operator, equivalently the image of the quadratic
$\mathfrak{osp}(m|2n)$-Casimir up to the conventional choice of sign; see
\cite[Section~5]{Coulembier2013Orthosymplectic} and
\cite[Section~3]{LavickaSmid2015}.  The Pizzetti form is
$\mathfrak{osp}(m|2n)$-invariant.  The Casimir actions in the two variables
are consequently contragredient, and its evenness gives
\[
 \mathsf B_k(\Omega P,Q)=\mathsf B_k(P,\Omega Q).
\]
Thus $\Omega$ is self-adjoint for this pairing.  On harmonic degree $j$ its
eigenvalue is
\begin{equation}\label{eq:exceptional-casimir-eigenvalue}
 \chi_j=j(j+M-2).
\end{equation}
At $M=-2u$ one has
\[
 \chi_j-\chi_\ell=(j-\ell)(j+\ell-2u-2).
\]
Thus two harmonic rows can have the same Casimir eigenvalue only when
$j=\ell$ or $j=2u+2-\ell$.  If two generalized eigenvalues $\alpha\ne\beta$ are distinct, choose
polynomials $A,B$ with
$A(z)(z-\alpha)^N+B(z)(z-\beta)^N=1$.  Moving these polynomials in $\Omega$
across the pairing shows that the corresponding generalized eigenspaces are
orthogonal.  This is the B\'ezout argument for a self-adjoint operator.  Each
nontrivial mirror pair is precisely the pair already combined inside one
generalized module
$\widetilde{\mathcal H}_\ell$.

Multiplication by $r$ acts as $-1$ under the sphere functional.  Hence a
regular row $r^{(k-\ell)/2}\mathcal H_\ell$ has the same degeneracy as its
harmonic form.  It is zero for $\ell\leq u$ and nondegenerate for
$\ell\geq u+1$.  By \Cref{lem:confluent-exceptional-pairing}, every
generalized row in the first sum of
\eqref{eq:exceptional-generalized-fischer} is nondegenerate.

A low harmonic degree $\ell\leq u$ remains as a separate regular row exactly
when its mirror $2u+2-\ell$ has not yet entered the degree-$k$ column, namely
when
\[
 \ell<2u+2-k.
\]
Because $k-\ell$ is even, this is equivalent to $\ell\leq2u-k$.  Thus the
radical is the orthogonal sum
\[
 \bigoplus_{\substack{0\leq\ell\leq2u-k\\k-\ell\text{ even}}}
 r^{(k-\ell)/2}\mathcal H_\ell.
\]
For $u+1\leq k\leq2u$, the degree $2u-k$ lies below the exceptional range,
so its ordinary Fischer decomposition gives
\[
 \bigoplus_{\substack{0\leq\ell\leq2u-k\\k-\ell\text{ even}}}
 r^{(k-\ell)/2}\mathcal H_\ell
 =r^{k-u}\mathcal P_{2u-k}.
\]
For $k\geq2u+1$ the index set is empty.  This proves
\eqref{eq:exceptional-radical-identity} and the quotient statement.
\end{proof}

In the collision range, the reflected low-degree residue form induces a
canonical nondegenerate form on the exact radical.

\begin{definition}\label{def:transported-radical-pairing}
For $u+1\le k\le2u$, put
\begin{equation}\label{eq:graded-confluent-dh}
 d=k-u,\qquad h=2u-k,\qquad
 \mathfrak R_{u,k}=r^d\mathcal P_h.
\end{equation}
Write
\[
 \mathsf B_{u,k}(P,Q):=\mathcal I_{-2u}^{\SSph}(PQ),
\]
and let $\overline{\mathsf B}_{u,k}$ be its descended form on
$\mathcal P_k/\mathfrak R_{u,k}$.  Because $m>0$, write $r=r_{\mathrm b}+r_{\mathrm f}$ with
$r_{\mathrm b}$ a nonzero bosonic quadratic polynomial and
$r_{\mathrm f}$ of fermionic degree two.  If $A\ne0$, choose the smallest
fermionic degree $p$ occurring in its exterior-algebra expansion.  The
fermionic-degree-$p$ component of $r^dA$ is
$r_{\mathrm b}^{d}A_p$, which cannot vanish in the bosonic polynomial ring.
Thus multiplication by $r^d$ is injective on $\mathcal P$.  The low-degree residue pairing therefore transports
unambiguously to the radical by
\begin{equation}\label{eq:transported-radical-pairing}
 \mathsf C_{u,k}(r^dA,r^dB):=\mathcal R_{u,h}(AB),
 \qquad A,B\in\mathcal P_h.
\end{equation}
\end{definition}

The following row-action formula makes the direct verification explicit.
For a scalar invariant kernel $K$, let
\[
 T_KP(x):=\mathcal I_{M,y}^{\SSph}\bigl(K(x,y)P(y)\bigr).
\]

\begin{lemma}[Homogeneous monomial contraction]
\label{lem:homogeneous-monomial-contraction}
Let $M$ be an integral superdimension, let $Q\in\mathcal P_\ell$, and let
$b\ge0$ have the same parity as $\ell$.  Put
\[
 N=\frac{b+\ell}{2},\qquad q_0=\frac{b-\ell}{2}.
\]
Then
\begin{equation}\label{eq:homogeneous-monomial-contraction}
 \mathcal I_{M,y}^{\SSph}\bigl(s^bQ(y)\bigr)
 =(-1)^N\frac{2\pi^{M/2}b!}{\Gamma(N+M/2)}
 \sum_{\nu=0}^{\lfloor\ell/2\rfloor}
 \frac{r^{q_0+\nu}(\Delta^\nu Q)(x)}
 {4^\nu\nu!(q_0+\nu)!}.
\end{equation}
Here a term with $q_0+\nu<0$ is interpreted as zero.  The reciprocal gamma
factor is understood by continuation at exceptional superdimensions.
\end{lemma}

\begin{proof}
The calculation is a formal consequence of the product rule and therefore is
valid in the polynomial superspace.  For formal parameters $z$ and $\lambda$,
one has
\begin{equation}\label{eq:formal-heat-contraction}
 \left.e^{z\Delta_y}\bigl(e^{\lambda s}Q(y)\bigr)\right|_{y=0}
 =e^{4z\lambda^2r}\bigl(e^{z\Delta}Q\bigr)(4z\lambda x).
\end{equation}
Indeed, let $D_{x,y}$ be the scalar directional derivative characterized by
\[
 e^{aD_{x,y}}Q(y)=Q(y+ax).
\]
Then
\[
 e^{-\lambda s}\Delta_y e^{\lambda s}
 =\Delta_y+4\lambda D_{x,y}+4\lambda^2r.
\]
The three operators on the right commute, and exponentiation followed by
setting $y=0$ gives \eqref{eq:formal-heat-contraction}.  Since $Q$ is
homogeneous, the coefficient of $z^N\lambda^b$ on the right-hand side of
\eqref{eq:formal-heat-contraction} is
\[
 \sum_{\nu=0}^{\lfloor\ell/2\rfloor}
 \frac{4^{N-\nu}r^{q_0+\nu}(\Delta^\nu Q)(x)}
 {\nu!(q_0+\nu)!}.
\]
Comparing with the coefficient on the left gives
\[
 \left.\Delta_y^N\bigl(s^bQ(y)\bigr)\right|_{y=0}
 =N!b!\sum_{\nu=0}^{\lfloor\ell/2\rfloor}
 \frac{4^{N-\nu}r^{q_0+\nu}(\Delta^\nu Q)(x)}
 {\nu!(q_0+\nu)!}.
\]
Multiplication by the homogeneous Pizzetti coefficient
\eqref{eq:pizzetti-homogeneous-coefficient} proves
\eqref{eq:homogeneous-monomial-contraction}.
\end{proof}

\begin{lemma}\label{lem:regular-row-action}
Let $M$ be an integral superdimension, let $0\le p\le\lfloor k/2\rfloor$,
and put $\ell=k-2p$.  If
\[
 K(x,y)=\sum_{j=0}^{\lfloor k/2\rfloor}c_j(rt)^js^{k-2j},
 \qquad H\in\mathcal H_\ell,
\]
then
\begin{equation}\label{eq:regular-row-action}
 T_K(t^pH)(x)=\Lambda_{M,k,p}(c)\,r^pH(x),
\end{equation}
where
\begin{align}\label{eq:regular-row-scalar}
 \Lambda_{M,k,p}(c)
 &=\sum_{j=0}^{p}c_j\Phi_{M,k}(p,j),\notag\\
 \Phi_{M,k}(p,j)
 &=(-1)^k\frac{2\pi^{M/2}(k-2j)!}
 {(p-j)!\,\Gamma(M/2+k-p-j)}.
\end{align}
At a negative even superdimension, the reciprocal gamma factor is interpreted
by analytic continuation; in particular, it vanishes at the nonpositive
integers.
\end{lemma}

\begin{proof}
Apply \Cref{lem:homogeneous-monomial-contraction} with
$Q=H$, $\ell=k-2p$, and $b=k-2j$.  Since $H$ is harmonic, only the term
$\nu=0$ remains.  It vanishes for $j>p$, while for $j\le p$ it equals
\[
 (-1)^{k-p-j}\frac{2\pi^{M/2}(k-2j)!}
 {(p-j)!\,\Gamma(M/2+k-p-j)}r^{p-j}H(x).
\]
Inside the Pizzetti functional the factor $t^{p+j}$ contributes
$(-1)^{p+j}$, and the exterior factor $r^j$ restores radial degree $p$.
The total sign is $(-1)^k$, which gives
\eqref{eq:regular-row-scalar}.  This is the monomial case of the
superspace Funk--Hecke formula, with the exceptional zeros supplied by the
continued reciprocal gamma factor.
\end{proof}

The quotient and high-degree inverse kernels can now be verified directly
at the fixed exceptional superdimension.

\begin{lemma}\label{lem:direct-exceptional-inverses}
Let the fixed superspace have $M=-2u$ and $k\ge u+1$.

If $u+1\le k\le2u$, set $d=k-u$ and $h=2u-k$, and define
\begin{equation}\label{eq:direct-quotient-kernel}
 K_{u,k}^{\mathrm{quo}}(x,y)
 =\gamma_{-2u,k}
 \sum_{j=0}^{d-1}
 (-1)^j\frac{k!}{j!(k-2j)!(d-j)_j}(rt)^js^{k-2j}.
\end{equation}
Then, for every $P\in\mathcal P_k$,
\begin{equation}\label{eq:direct-quotient-reproduction}
 \pi_x\,\mathcal I_{-2u,y}^{\SSph}
 \bigl(K_{u,k}^{\mathrm{quo}}(x,y)P(y)\bigr)=\pi_xP(x),
\end{equation}
where $\pi_x:\mathcal P_k\to\mathcal P_k/r^d\mathcal P_h$ is the quotient map.

If $k\ge2u+1$, the coefficient
$K_{u,k}^{\mathrm{high}}=[\rho^k]\mathscr L_u$ satisfies
\begin{equation}\label{eq:direct-high-reproduction}
 \mathcal I_{-2u,y}^{\SSph}
 \bigl(K_{u,k}^{\mathrm{high}}(x,y)P(y)\bigr)=P(x).
\end{equation}
\end{lemma}

\begin{proof}
Put $d=k-u$.  In the collision range set $J=d$, while in the high-degree
range set $J=\lfloor k/2\rfloor+1$.  In both cases the coefficients of the
kernel under consideration can be written uniformly as
\begin{equation}\label{eq:exceptional-unified-coefficients}
 c_j^*=\frac{(-1)^{k+j}\pi^u\Gamma(d-j)}
 {2j!(k-2j)!},\qquad 0\le j<J,
\end{equation}
with $c_j^*=0$ for $j\ge J$.  For the collision kernel this follows from
$(d-j)_j=\Gamma(d)/\Gamma(d-j)$ and the value of
$\gamma_{-2u,k}$; in the high-degree range it is the coefficientwise
specialization of \eqref{eq:global-coefficient-formula}.  Let $T^*$ be the
integral operator defined by these coefficients.

We first treat an ordinary harmonic row $t^p\mathcal H_{k-2p}$.  By
\Cref{lem:regular-row-action}, if $p<d$ then
\begin{align}\label{eq:exceptional-regular-binomial}
 \Lambda_{-2u,k,p}(c^*)
 &=\sum_{j=0}^{p}(-1)^j
 \frac{\Gamma(d-j)}{j!(p-j)!\Gamma(d-p-j)}\notag\\
 &=\sum_{j=0}^{p}(-1)^j
 \binom pj\binom{d-j-1}{p}=1.
\end{align}
For the last equality, take the coefficient of $z^p$ after summing
$(1+z)^{d-j-1}$ against $(-1)^j\binom pj$; the result is
$z^p(1+z)^{d-p-1}$.  If $p\ge d$, every reciprocal gamma factor in
\eqref{eq:regular-row-scalar} vanishes for $0\le j<d$, and hence
$\Lambda_{-2u,k,p}(c^*)=0$.  Thus the collision operator is the identity on
the regular rows outside the radical and vanishes on the regular rows
contained in it; in the high-degree range all regular rows are reproduced.

It remains to verify the generalized rows without an implicit elimination
argument.  Let $r^p\widetilde{\mathcal H}_\ell$ occur in
\eqref{eq:exceptional-generalized-fischer} and put
\[
 h=2u+2-\ell,\qquad a=\ell-u-1=d-2p-1\ge1.
\]
For $H\in\mathcal H_h$, choose $Q_H$ as in
\eqref{eq:exceptional-lift-equation} and put $b_H=r^aH$.  Iteration of the
radial transition identity gives
\begin{equation}\label{eq:all-lift-iterates}
 \Delta^\nu Q_H=
 \begin{cases}
 Q_H,&\nu=0,\\[2mm]
 (-4)^{\nu-1}
 \left(\dfrac{(a-1)!}{(a-\nu)!}\right)^2r^{a-\nu}H,
 &1\le\nu\le a,\\[2mm]
 0,&\nu>a.
 \end{cases}
\end{equation}
Apply \Cref{lem:homogeneous-monomial-contraction} term by term to
$T^*(r^pQ_H)$.  Since $d=a+2p+1$, all contributions have the form
\begin{equation}\label{eq:generalized-explicit-action}
 T^*(r^pQ_H)=\alpha_p\,r^pQ_H+\beta_p\,r^pb_H,
\end{equation}
where
\begin{align}
 \alpha_p
 &=(-1)^k2\pi^{-u}\sum_{j=0}^{p}
 c_j^*\frac{(k-2j)!}{(p-j)!\Gamma(d-p-j)},
 \label{eq:generalized-alpha}\\
 \beta_p
 &=(-1)^k\frac{\pi^{-u}}2
 \sum_{j=0}^{a+p}c_j^*\frac{(k-2j)!}
 {\Gamma(a+p+1-j)}\notag\\
 &\qquad\times
 \sum_{\nu=\max(1,j-p)}^{a}
 \frac{(-1)^{\nu-1}((a-1)!)^2}
 {\nu!(p-j+\nu)!((a-\nu)!)^2}.
 \label{eq:generalized-beta}
\end{align}
The upper limit satisfies $a+p=d-p-1<d=J$ in the collision range.  In the
high-degree range, $\ell\le2u+2$ implies
$a+p=(k+\ell-2u-2)/2\le\lfloor k/2\rfloor<J$.  Thus every coefficient used in
this calculation is present in both regimes.  Equation
\eqref{eq:exceptional-regular-binomial} gives $\alpha_p=1$.

We now show directly that $\beta_p=0$.  Substitute
\eqref{eq:exceptional-unified-coefficients} into
\eqref{eq:generalized-beta}, interchange the finite sums, and write
\begin{equation}\label{eq:beta-finite-difference}
 \beta_p=\frac{((a-1)!)^2}{4}
 \sum_{\nu=1}^{a}
 \frac{(-1)^{\nu-1}}{\nu!((a-\nu)!)^2}\,I_\nu,
\end{equation}
where
\[
 I_\nu=\sum_{j=0}^{p+\nu}
 (-1)^j\frac{\Gamma(d-j)}
 {j!\,\Gamma(a+p+1-j)\,(p-j+\nu)!}.
\]
With $\kappa=p+\nu-j$ and $d=a+2p+1$, this becomes
\begin{equation}\label{eq:finite-difference-inner-sum}
 I_\nu=\frac{(-1)^{p+\nu}p!}{(p+\nu)!}
 \sum_{\kappa=0}^{p+\nu}(-1)^{\kappa}\binom{p+\nu}{\kappa}
 \binom{a+p-\nu+\kappa}{p}.
\end{equation}
The last sum is the $(p+\nu)$-th finite difference of a polynomial in $\kappa$ of
degree $p$.  Since $\nu\ge1$, it is zero.  Therefore every $I_\nu$ vanishes,
so $\beta_p=0$.  Equations \eqref{eq:generalized-explicit-action}--
\eqref{eq:finite-difference-inner-sum} reproduce the lift representatives
$r^pQ_H$.  The regular-row calculation already reproduces the harmonic
submodule $r^p\mathcal H_\ell$, including its socle $r^pb_H$; together these
vectors span $r^p\widetilde{\mathcal H}_\ell$.  Hence $T^*$ is the identity on
every generalized harmonic row, including its head--socle extension.

Suppose first that $u+1\le k\le2u$.  Here $J=d$.  Since
$h=2u-k\le u$, the ordinary Fischer decomposition of $\mathcal P_h$ has no
exceptional generalized rows; multiplication by $r^d$ identifies its rows
with the regular rows of radial level $p\ge d$ in $\mathcal P_k$.  The
preceding calculation therefore shows that $T^*P-P\in r^d\mathcal P_h$ for
every $P\in\mathcal P_k$, which is exactly
\eqref{eq:direct-quotient-reproduction}.  Formula
\eqref{eq:exceptional-unified-coefficients} is equivalent to
\eqref{eq:direct-quotient-kernel}.  The finite-part coefficient has these
same terms for $j<d$, whereas every remaining term contains $r^d$ in the
$x$-variable.  Hence $K_{u,k}^{\mathrm{quo}}$ is a representative of the
finite-part inverse kernel on the quotient.

If $k\ge2u+1$, then $d>\lfloor k/2\rfloor$ and every regular or generalized
row is reproduced.  Their direct sum is $\mathcal P_k$, so $T^*=I$ and
\eqref{eq:direct-high-reproduction} follows.  Since the polynomial residue
has degree $2u<k$, the coefficients
\eqref{eq:exceptional-unified-coefficients} are precisely those of
$[\rho^k]\mathscr L_u$.
\end{proof}

The preceding results identify the homogeneous coefficients of the two
Laurent layers with the inverse kernels of the associated graded exceptional
Fischer filtration.

\begin{theorem}\label{thm:confluent-supersphere-duality}
For a formal series $F(\rho)$, write $[\rho^k]F$ for its
$\rho^k$-coefficient.  Let
\[
 \mathsf G_{u,k}^{[-1]}=[\rho^k]\mathscr R_u,
 \qquad
 \mathsf G_{u,k}^{[0]}=[\rho^k]\mathscr L_u.
\]
Then the exceptional homogeneous spaces fall into the following three
regimes.
\begin{enumerate}[label=\textup{(\roman*)}]
\item If $0\le k\le u$, then
\begin{equation}\label{eq:graded-confluent-low}
 \mathsf G_{u,k}^{[-1]}
 =\widehat\gamma_{u,k}H_{-2u,k}^{\mathrm{res}},
\end{equation}
and it is the inverse kernel of the nondegenerate residue pairing
$\mathcal R_{u,k}$ on the full space $\mathcal P_k$.

\item If $u+1\le k\le2u$, put $d=k-u$ and $h=2u-k$.  The residue coefficient
is
\begin{equation}\label{eq:confluent-radical-kernel}
 \mathsf G_{u,k}^{[-1]}(x,y)
 =(rt)^d\widehat\gamma_{u,h}
 H_{-2u,h}^{\mathrm{res}}(x,y),
\end{equation}
so it is the inverse kernel of $\mathsf C_{u,k}$ on the exact radical
$\mathfrak R_{u,k}=r^d\mathcal P_h$.  The finite part descends to the inverse
kernel of the ordinary Pizzetti form on
$\mathcal P_k/\mathfrak R_{u,k}$; one representative is
\begin{equation}\label{eq:graded-confluent-quotient-kernel}
 \mathsf G_{u,k}^{\mathrm{quo}}(x,y)
 =\gamma_{-2u,k}
 \sum_{j=0}^{d-1}
 (-1)^j\frac{k!}{j!(k-2j)!(k-u-j)_j}
 (rt)^js^{k-2j},
\end{equation}
where
$\gamma_{-2u,k}=(-1)^k\pi^u\Gamma(k-u)/(2k!)$.
Consequently
\begin{equation}\label{eq:graded-confluent-pairing}
 \overline{\mathsf B}_{u,k}\oplus\mathsf C_{u,k}
\end{equation}
is nondegenerate on
\begin{equation}\label{eq:graded-confluent-space}
 (\mathcal P_k/\mathfrak R_{u,k})\oplus\mathfrak R_{u,k}.
\end{equation}

\item If $k\ge2u+1$, then $\mathsf G_{u,k}^{[-1]}=0$ and
$\mathsf G_{u,k}^{[0]}$ is the ordinary inverse Pizzetti kernel on the full
space $\mathcal P_k$.
\end{enumerate}
\end{theorem}

\begin{proof}
The low-degree statement is \Cref{thm:exceptional-residue-reproduction}.
In the collision range, every monomial of degree $k$ in the polynomial $D^u$
contains at least $(rt)^d$.  Writing its index as $j=d+\ell$ gives
\[
 [\rho^k]\mathscr R_u
 =(rt)^d\frac{(-1)^u\pi^u}{h!d!}
 H_{-2u,h}^{\mathrm{res}}
 =(rt)^d\widehat\gamma_{u,h}H_{-2u,h}^{\mathrm{res}},
\]
because $u-h=d$.  This is the inverse of the transported radical pairing by
\Cref{thm:exceptional-residue-reproduction}.  The quotient formula and its
reproducing property are exactly
\Cref{lem:direct-exceptional-inverses}; in particular, no generic
reproducing identity is specialized inside a fixed coordinate realization.
For $k\ge2u+1$, the polynomial residue has no coefficient and the direct
high-degree identity in the same lemma proves that the logarithmic coefficient
is the ordinary inverse kernel.
\end{proof}

It is important to distinguish the transported radical form from a
naive derivative of the fixed-degree Pizzetti coefficient.

\begin{remark}
The radical pairing is not obtained by differentiating only the scalar
Pizzetti coefficient in degree $k$ while holding the fixed super Laplacian
constant; that derivative still vanishes on the radical.  It is selected
canonically by the residue of the inverse kernel, equivalently by reflection
to degree $h=2u-k$, application of the low-degree residue pairing, and
transport by $r^{k-u}$.
\end{remark}

\section{Application to radial \texorpdfstring{$q$}{q}-monogenic projection}\label{sec:q-application}

We now use the intrinsic radial $q$-vector derivative of
\cite{BarseghyanBoryReyesSchneiderZhang2026}.  The specialization at integral
superdimension and the faithful finite superspace realization used below are
those of \cite{BoryReyesSchneiderBarseghyanZhang2026}.  We recall the necessary two-point rules, derive the projected zonal
symbol, and then construct finite-block tensor representatives of the
projection by a dual-basis correction.  The latter correction is necessary because the
one-sided radial projector does not commute with supersphere contraction.  Fix an integral superdimension
$M$.  Let $q$ be an indeterminate and write
\[
 [a]_q=\frac{1-q^a}{1-q};
\]
when numerical inequalities are invoked we specialize to $0<q<1$.  Let
$\mathcal R_M(x,y)$ be the two-vector radial algebra at superdimension $M$,
work over $\mathbb K=\mathbb C(q)$, and set
\[
 r=x^2,\qquad s=\{x,y\},\qquad t=y^2,\qquad
 \omega=x\wedge y.
\]
The even bidegree-$(k,k)$ zonal module is
\begin{equation}\label{eq:zonal-module}
 \mathcal Z_{M,k}
 =\bigoplus_{j=0}^{\lfloor k/2\rfloor}\mathbb K(rt)^js^{k-2j}
 \oplus
 \bigoplus_{j=0}^{\lfloor(k-1)/2\rfloor}
 \mathbb K\omega(rt)^js^{k-1-2j}.
\end{equation}
This normal form follows from the scalar--bivector decomposition of a
two-vector radial algebra together with the bidegree constraints; see
\cite{Sommen1997,DeSchepperGuzmanSommen2017}.

For $A\in\mathbb K[r,s,t]$, put
\begin{equation}\label{eq:scalar-qderivative}
 \partial_{x,q}^{\mathrm{sc}}A
 =(1+q)x\,\delta_{q^2}^{(r)}A
 +2T_{q^2}^{(r)}y\,\delta_q^{(s)}A,
\end{equation}
where $T_P^{(z)}f(z)=f(Pz)$ and
$\delta_P^{(z)}f=(f(z)-f(Pz))/((1-P)z)$.  These are the standard Jackson
shift and difference operators; see \cite{KacCheung2002}.  On the two zonal sectors,
\begin{align}
 \partial_{x,q}^{M,\mathrm R}(A)&=\partial_{x,q}^{\mathrm{sc}}A,
 \label{eq:qder-scalar}\\
 \partial_{x,q}^{M,\mathrm R}(\omega B)
 &=[M-1]_q\,yB+q^{M-1}\omega\,\partial_{x,q}^{\mathrm{sc}}B.
 \label{eq:qder-bivector}
\end{align}

For clarity, we state the precise finite-block Fischer result imported from
the companion papers.

\begin{proposition}[Finite localized $q$-Fischer input]
\label{prop:imported-q-fischer}
Let $Y_0=\{y_1,\ldots,y_N\}$ be finite, and let
$\mathcal V_{M,k}^{Y_0}$ be the $x$-homogeneous degree-$k$ component of the
specialized radial algebra generated by $x$ and $Y_0$.  After extension of scalars from $\mathbb R(q)$ to $\mathbb C(q)$, it is a
finite free module over the passive scalar ring
\[
 \mathcal T_M^{Y_0}
 =\mathbb C(q)[c_{ij}:1\le i\le j\le N],
 \qquad c_{ij}=\tfrac12\{y_i,y_j\}.
\]
Define
\[
 \mathcal K_{M,k}^{Y_0}
 =\partial_{x,q}^{M,Y_0,\mathrm R}L_x
   \big|_{\mathcal V_{M,k}^{Y_0}},
 \qquad
 \mathcal D_{M,k}^{Y_0}=\det\mathcal K_{M,k}^{Y_0}.
\]
After localization by $\mathcal D_{M,k}^{Y_0}$,
\begin{equation}\label{eq:imported-fischer-decomposition}
 \mathcal V_{M,k+1}^{Y_0}
 =\Ker\!\left(
 \partial_{x,q}^{M,Y_0,\mathrm R}
 \big|_{\mathcal V_{M,k+1}^{Y_0}}
 \right)
 \oplus L_x\mathcal V_{M,k}^{Y_0},
\end{equation}
and the projection onto the first summand is
\begin{equation}\label{eq:imported-fischer-projector}
 \Pi_{M,k+1}^{(q),Y_0}
 =I-L_x(\mathcal K_{M,k}^{Y_0})^{-1}
 \partial_{x,q}^{M,Y_0,\mathrm R}.
\end{equation}
For degree zero we set $\Pi_{M,0}^{(q),Y_0}=I$.  Under a numerical
specialization $q=q_0\in(0,1)$, the same formulas hold whenever the evaluated
determinant is nonzero.  If the fixed superspace has bosonic dimension $m$
and $|\{x\}\cup Y_0|\le m$, its standard coordinate representation is
faithful, so the projector has an unambiguous coordinate realization.
\end{proposition}

\begin{proof}
After extension of scalars, the decomposition and projector formula are
\cite[Theorem~5.5]{BoryReyesSchneiderBarseghyanZhang2026}; their universal
formal-dimension version is
\cite[Theorem~5.3]{BarseghyanBoryReyesSchneiderZhang2026}.  The sufficient
faithfulness range is
\cite[Proposition~3.4]{BoryReyesSchneiderBarseghyanZhang2026}.
\end{proof}

For the two-vector block $Y_0=\{y\}$, write
$\Pi_{M,k}^{(q)}=\Pi_{M,k}^{(q),\{y\}}$.  For $k=0$ set
$Z_{M,0}^{(q)}=1$.  The following theorem defines $Z_{M,k}^{(q)}$ for
$k\ge1$ directly in the smaller zonal localization and then identifies it
with the Fischer projection in the common localization.  Henceforth
assume $k\geq1$ and, for compactness, write
\begin{align}
 \lambda_{M,k}(q)&=[M-1]_q-q^{M-1}[k-1]_q,\\
 d_{M,k}(q)&=\lambda_{M,k}(q)+[k]_q,\\
 \theta_{M,k,j}(q)&=\lambda_{M,k}(q)-q^{M+2j-1}[k-2j]_q.
\end{align}

We first solve the two-point right $q$-monogenic projection problem in
closed form.

\begin{theorem}\label{thm:closed-recurrence}
Let $k\geq1$.  After localization by
\begin{equation}\label{eq:zonal-denominator}
 \mathfrak d_{M,k}^{\mathrm{zon}}(q)
 =d_{M,k}(q)
  \prod_{j=1}^{\lfloor(k-1)/2\rfloor}\theta_{M,k,j}(q),
\end{equation}
there is a unique element, denoted $Z_{M,k}^{(q)}\in\mathcal Z_{M,k}$,
such that
\begin{equation}\label{eq:zonal-characterization}
 \partial_{x,q}^{M,\mathrm R}Z_{M,k}^{(q)}=0,
 \qquad Z_{M,k}^{(q)}\equiv s^k\pmod{x\mathcal R_M(x,y)}.
\end{equation}
In the common localization in which the full Fischer determinant
$\mathcal D_{M,k-1}^{\{y\}}$ is also inverted,
\begin{equation}\label{eq:Z-def}
 Z_{M,k}^{(q)}(x,y)=\Pi_{M,k}^{(q)}(s^k).
\end{equation}
Writing
\begin{equation}\label{eq:Z-coeff-expansion}
 Z_{M,k}^{(q)}
 =\sum_j a_{k,j}^{M,q}(rt)^js^{k-2j}
 +\sum_j b_{k,j}^{M,q}\omega(rt)^js^{k-1-2j},
\end{equation}
one has
\begin{align}
 a_{k,0}^{M,q}&=\frac{\lambda_{M,k}(q)}{d_{M,k}(q)},&
 b_{k,0}^{M,q}&=-\frac{2[k]_q}{d_{M,k}(q)},\label{eq:initial-coeffs}\\
 b_{k,j}^{M,q}
 &=\frac{4q^{M+4j-3}[k-2j]_q[k+1-2j]_q}
 {[2j]_q\theta_{M,k,j}(q)}b_{k,j-1}^{M,q},
 &&1\le j\le\left\lfloor\frac{k-1}{2}\right\rfloor,
 \label{eq:b-recurrence}\\
 a_{k,j}^{M,q}
 &=-\frac{\lambda_{M,k}(q)}{2q^{2j}[k-2j]_q}b_{k,j}^{M,q},
 &&0\le j\le\left\lfloor\frac{k-1}{2}\right\rfloor.
 \label{eq:a-recurrence}
\end{align}
At even degree $k=2\ell$, the final scalar coefficient is
\begin{equation}\label{eq:even-endpoint}
 a_{k,\ell}^{M,q}
 =-\frac{2q^{M+k-3}}{[k]_q}b_{k,\ell-1}^{M,q}.
\end{equation}
Moreover, for $0\le j\le\lfloor(k-1)/2\rfloor$,
\begin{equation}\label{eq:b-product}
 b_{k,j}^{M,q}
 =-\frac{2[k]_q}{d_{M,k}(q)}
 \frac{4^jq^{2j^2+(M-1)j}
 \prod_{h=1}^j[k-2h]_q[k+1-2h]_q}
 {\prod_{h=1}^j[2h]_q\theta_{M,k,h}(q)}.
\end{equation}
\end{theorem}

\begin{proof}
Put
\[
 S_{k,j}=(rt)^js^{k-2j},\qquad
 W_{k,j}=\omega(rt)^js^{k-1-2j}.
\]
The rules \eqref{eq:qder-scalar}--\eqref{eq:qder-bivector} give
\begin{align}
 \partial_{x,q}^{M,\mathrm R}S_{k,j}
 &=[2j]_qxr^{j-1}t^js^{k-2j}
 +2q^{2j}[k-2j]_qy(rt)^js^{k-1-2j},\label{eq:zonal-S-action}\\
 \partial_{x,q}^{M,\mathrm R}W_{k,j}
 &=q^{M-1}\frac{[2j]_q}{2}xr^{j-1}t^js^{k-2j}
 +2q^{M+2j-1}[k-1-2j]_qxr^jt^{j+1}s^{k-2-2j}\notag\\
 &\quad+\lambda_{M,k}(q)y(rt)^js^{k-1-2j}.
 \label{eq:zonal-W-action}
\end{align}
Insert $Z=\sum_ja_jS_{k,j}+\sum_jb_jW_{k,j}$.  At level $j$ the coefficients
of the two independent vector monomials are
\begin{align}
 2q^{2j}[k-2j]_qa_j+\lambda_{M,k}(q)b_j&=0,\label{eq:zonal-y-system}\\
 [2j]_qa_j+q^{M-1}\frac{[2j]_q}{2}b_j
 +2q^{M+2j-3}[k+1-2j]_qb_{j-1}&=0.
 \label{eq:zonal-x-system}
\end{align}
Their determinant is
$-[2j]_q\theta_{M,k,j}(q)$.  Modulo the left ideal $x\mathcal R_M(x,y)$ one has
$\omega=x\wedge y=xy-s/2\equiv-s/2$.  Hence the normalization in
\eqref{eq:zonal-characterization} is $a_0-b_0/2=1$, which, together with
\eqref{eq:zonal-y-system} at $j=0$, gives
\eqref{eq:initial-coeffs}.  Solving
\eqref{eq:zonal-y-system}--\eqref{eq:zonal-x-system} successively gives
\eqref{eq:b-recurrence}--\eqref{eq:a-recurrence}.  At even degree the last
bivector coefficient is absent, and the remaining scalar equation is
\eqref{eq:even-endpoint}; hence no extra terminal denominator occurs.
Iteration gives \eqref{eq:b-product}.  The pivots listed in
\eqref{eq:zonal-denominator} prove existence and uniqueness in the zonal
localization.  In the common localization where
$\mathcal D_{M,k-1}^{\{y\}}$ is also inverted, the determinant-localized
Fischer decomposition of \cite{BarseghyanBoryReyesSchneiderZhang2026}
provides the unique null representative of the same congruence class.
Therefore it agrees with $Z_{M,k}^{(q)}$, proving \eqref{eq:Z-def}.
\end{proof}

Let $\mathfrak r$ denote reversion and let $\tau$ interchange $x$ and
$y$.  Define the left derivative in the $y$-variable by
$\partial_{y,q}^{M,\mathrm L}
 =\mathfrak r\partial_{y,q}^{M,\mathrm R}\mathfrak r$.
On a faithful coordinate realization, an orthosymplectic transformation $g$
is called admissible for the chosen finite block if its simultaneous linear
action on all represented vector variables maps that block to itself.  It
then induces an automorphism $\rho(g)$ of the radial coefficient and value
module.

The projected zonal symbol is null in both variables, with opposite
handedness, and is equivariant under simultaneous admissible
orthosymplectic transformations.

\begin{theorem}\label{thm:q-symmetry}
The projected zonal symbol satisfies
\begin{equation}\label{eq:two-sided-proved}
 \partial_{x,q}^{M,\mathrm R}Z_{M,k}^{(q)}=0,
 \qquad
 \partial_{y,q}^{M,\mathrm L}Z_{M,k}^{(q)}=0.
\end{equation}
On every faithful two-supervector realization it is covariant under every
admissible simultaneous orthosymplectic transformation $g$.  Here $\rho(g)$
denotes the induced action on the
radial coefficient/value module (conjugation on Clifford values together with
the natural change of vector variables):
\begin{equation}\label{eq:kernel-covariance}
 Z_{M,k}^{(q)}(gx,gy)=\rho(g)Z_{M,k}^{(q)}(x,y).
\end{equation}
\end{theorem}

\begin{proof}
The coefficients in \eqref{eq:Z-coeff-expansion} depend on $r$ and $t$
only through $rt$.  Reversion and interchange of $x$ and $y$ both change the
sign of $\omega$, while leaving $s$ and $rt$ unchanged.  Hence
$\mathfrak rZ=\tau Z$.  Relabelling covariance of the right derivative gives
\[
 \partial_{y,q}^{M,\mathrm R}(\tau Z)
 =\tau\bigl(\partial_{x,q}^{M,\mathrm R}Z\bigr)=0,
\]
and conjugation by reversion yields the second equation in
\eqref{eq:two-sided-proved}.  A simultaneous admissible superrotation
preserves $r$, $s$, and $t$ and acts naturally on $\omega$ and on the value
module.  Applying this action term by term in
\eqref{eq:Z-coeff-expansion} proves \eqref{eq:kernel-covariance}.
\end{proof}

\subsection{Finite-block tensor representatives by ambient duality}

The projected zonal symbol of the preceding subsection is a distinguished
right/left null covariant, but one-sided application of the radial Fischer
projector does not transfer the Pizzetti reproducing identity.

\begin{proposition}[Degree-one failure of the one-sided transfer]
\label{prop:one-sided-failure}
Let $M\ge2$, $0<q<1$, $U=\varnothing$, and $k=1$.  Put
$\sigma_M=2\pi^{M/2}/\Gamma(M/2)$ and define the one-sided candidate
\[
 \mathscr G_{M,1}^{\mathrm{one}}(x,y)[a]
 =\Pi_{M,1}^{(q),\{y\}}
   \bigl(\mathsf G_{M,1}(x,y)a\bigr).
\]
For $P(y)=y$ one has
\begin{equation}\label{eq:one-sided-counterexample}
 \mathcal I_{M,y}^{\SSph}\!\left(
 \mathscr G_{M,1}^{\mathrm{one}}(x,y)[y]\right)
 =\left(1-\frac{M}{[M]_q}\right)x\ne0,
 \qquad
 \Pi_{M,1}^{(q),\{y\}}x=0.
\end{equation}
Thus the projector cannot be moved through supersphere contraction in the
manner used by the naive one-sided construction.
\end{proposition}

\begin{proof}
The generic degree-one inverse kernel is
\[
 \mathsf G_{M,1}(x,y)=-\frac{M}{2\sigma_M}\,s,
 \qquad s=\{x,y\}.
\]
Writing $t=y^2$, the degree-one projector gives
\[
 \Pi_{M,1}^{(q),\{y\}}(s\,y)
 =s\,y-\frac{2}{[M]_q}x\,t.
\]
The invariant second moments are
\[
 \mathcal I_{M,y}^{\SSph}(s\,y)=-\frac{2\sigma_M}{M}x,
 \qquad
 \mathcal I_{M,y}^{\SSph}(t)=-\sigma_M.
\]
Substitution yields the first identity in
\eqref{eq:one-sided-counterexample}.  The second follows directly from the
one-variable degree-one Fischer decomposition.  Since $[M]_q<M$ for
$M\ge2$ and $0<q<1$, the defect is nonzero.
\end{proof}

A finite-block tensor representative of the projection map is obtained through
the duality supplied by the relevant nondegenerate scalar pairing.  We write
$\jmath_{y\to x}$ for the coefficient-preserving relabelling of the active
variable $y$ as $x$, and $\jmath_{x\to y}$ for its inverse.

\begin{lemma}[Finite-block coevaluation]
\label{lem:finite-block-coevaluation}
Let $\mathbb F$ be a field of characteristic zero, let $V$ be a finite-dimensional
$\mathbb F$-vector space with a nondegenerate bilinear form $\mathsf B$, and
let $E$ and $W$ be finite-dimensional $\mathbb F$-vector spaces.  Extend $\mathsf B$ to the
nondegenerate coefficient-dual pairing
\begin{equation}\label{eq:coefficient-dual-pairing}
 \mathsf B^E(\phi\otimes\lambda,p\otimes a)
 =\mathsf B(\phi,p)\lambda(a),
 \qquad
 \phi,p\in V,\quad \lambda\in E^*,\quad a\in E.
\end{equation}
Let $S_y\subseteq V_y\otimes E$ be a finite subspace, put
$S_x=\jmath_{y\to x}S_y$, and let $A:S_x\to W$ be linear.  Choose a basis
$\{P_\alpha\}$ of $S_y$ and an ambient dual family
$\{P^\alpha\}\subseteq V_y\otimes E^*$ satisfying
\[
 \mathsf B^E(P^\alpha,P_\beta)=\delta^\alpha_\beta.
\]
Such a family exists by the nondegeneracy of
\eqref{eq:coefficient-dual-pairing}.  Define
\begin{equation}\label{eq:finite-block-coevaluation}
 \mathscr K_A(x,y)
 =\sum_\alpha A\bigl(\jmath_{y\to x}P_\alpha\bigr)
   \otimes P^\alpha(y)
 \in W\otimes(V_y\otimes E^*).
\end{equation}
Then, for every $P\in S_y$,
\begin{equation}\label{eq:finite-block-reproduction}
 \bigl\langle\mathscr K_A,P\bigr\rangle_{\mathsf B^E,y}
 =A\bigl(\jmath_{y\to x}P\bigr).
\end{equation}
The represented operator in \eqref{eq:finite-block-reproduction} is
independent of all choices.  The tensor
\eqref{eq:finite-block-coevaluation} is unique modulo
$W\otimes S_y^{\perp}$, where $S_y^{\perp}$ is the annihilator of $S_y$
under \eqref{eq:coefficient-dual-pairing}; when $S_y=V_y\otimes E$, it is
the canonical coevaluation tensor and is basis independent.
\end{lemma}

\begin{proof}
Write $P=\sum_\beta c_\beta P_\beta$.  Contracting
\eqref{eq:finite-block-coevaluation} with $P$ and using the defining duality
gives
\[
 \sum_{\alpha,\beta}c_\beta
 \mathsf B^E(P^\alpha,P_\beta)
 A(\jmath_{y\to x}P_\alpha)
 =\sum_\beta c_\beta A(\jmath_{y\to x}P_\beta).
\]
Changing the ambient dual family adds a tensor whose second factor annihilates
$S_y$, proving the final assertion.
\end{proof}

We apply this lemma after extending the passive scalar ring to a field.  For
$U=\{u_1,\ldots,u_N\}$ put
\begin{align}\label{eq:passive-fraction-field}
 \mathcal T_M^U
 &=\mathbb C(q)[c_{ij}:1\le i\le j\le N],
 & c_{ij}&=\tfrac12\{u_i,u_j\},\notag\\
 \mathbb F_U
 &=\operatorname{Frac}(\mathcal T_M^U)
 =\mathbb C(q)(c_{ij}).
\end{align}
Extend $\mathcal V_{M,k}^{U}$ and the localized Fischer projector from
$\mathcal T_M^U$ to $\mathbb F_U$.  Fix a faithful coordinate realization of
$\{x,y\}\cup U$.  In the sufficient range $N+2\le m$, the radial
coordinate map is injective; it identifies $\mathcal T_M^U$ with a subring of
the passive-coordinate scalar ring and hence embeds $\mathbb F_U$ into its
fraction field.  Let $S_{M,k}^{U}(y)$ be the coordinate image over
$\mathbb F_U$ of
$\mathcal V_{M,k}^{U}\otimes_{\mathcal T_M^U}\mathbb F_U$, with the active
variable relabelled as $y$.  This is a finite-dimensional
$\mathbb F_U$-vector space.  Let $E_U$ be the finite-dimensional
$\mathbb F_U$-subspace of the represented Clifford value algebra spanned by
the coefficients of $S_{M,k}^{U}(y)$ and of its image under the active Fischer
projector.  Thus
$S_{M,k}^{U}(y)\subseteq\mathcal P_k(y)\otimes E_U$.  Put
$S_{M,k}^{U}(x)=\jmath_{y\to x}S_{M,k}^{U}(y)$.

For a numerical specialization $q=q_0\in(0,1)$, first evaluate the finite
projector formulas at $q_0$, assuming that every required denominator and the
determinant remain nonzero, and then extend the passive scalar ring to
$\mathbb F_{U,q_0}=\mathbb C(c_{ij})$.  With $q$ kept generic, the statements
below also remain valid after a field specialization of the localized passive
scalar ring for which the induced map is injective on the chosen finite block.
The active $x$-projector uses $U$, but not the integrated variable $y$, as
passive support:
\[
 A_{M,k}^{(q),U}
 :=\Pi_{M,k}^{(q),U}\big|_{S_{M,k}^{U}(x)}:
 S_{M,k}^{U}(x)\longrightarrow
 \mathcal M_{M,k}^{(q),U},
 \qquad
 \mathcal M_{M,k}^{(q),U}
 =\Pi_{M,k}^{(q),U}S_{M,k}^{U}(x).
\]
This distinction is essential: inserting $y$ as a passive parameter produces
the failure in \Cref{prop:one-sided-failure}.

\begin{theorem}[Finite-block tensor representatives of the $q$-Fischer projection]
\label{thm:q-fischer-transfer}
Work over $\mathbb F_U$, over $\mathbb F_{U,q_0}$ after the numerical
procedure above, or over a field $\mathbb F$ obtained by an injective
finite-block specialization with $q$ generic; denote the chosen field by
$\mathbb F$.
For $k\ge1$, assume that the image of $\mathcal D_{M,k-1}^{U}$ is nonzero in
$\mathbb F$ and has been inverted; for $k=0$ use the identity projector.
Assume also that the coordinate realization of $\{x,y\}\cup U$ is faithful;
the sufficient standard condition is $N+2\le m$.
For every scalar form below, the contraction is taken with its coefficient-dual
extension from \eqref{eq:coefficient-dual-pairing} to the value block $E_U$.

\begin{enumerate}[label=\textup{(\roman*)}]
\item If $M\notin-2\mathbb N_0$, let
$\mathsf B_{M,k}(P,Q)=\mathcal I_{M}^{\SSph}(PQ)$ on $\mathcal P_k$.  Any
coevaluation representative
$\mathscr K_{M,k}^{(q),U}$ obtained from
\Cref{lem:finite-block-coevaluation} with $A=A_{M,k}^{(q),U}$ satisfies
\begin{equation}\label{eq:generic-q-reproduction}
 \bigl\langle\mathscr K_{M,k}^{(q),U},P\bigr\rangle_{
 \mathsf B_{M,k}^{E_U},y}
 =A_{M,k}^{(q),U}(P(x,U))
\end{equation}
for every $P\in S_{M,k}^{U}(y)$.

\item If $M=-2u$ and $0\le k\le u$, use the nondegenerate residue pairing
$\mathcal R_{u,k}$.  The resulting representative
$\mathscr K_{u,k}^{(q),[-1],U}$ satisfies
\begin{equation}\label{eq:q-low-residue-reproduction}
 \bigl\langle\mathscr K_{u,k}^{(q),[-1],U},P
 \bigr\rangle_{\mathcal R_{u,k}^{E_U},y}
 =A_{-2u,k}^{(q),U}(P(x,U)).
\end{equation}

\item If $M=-2u$ and $k\ge2u+1$, use the ordinary nondegenerate Pizzetti
pairing $\mathsf B_{-2u,k}$.  The resulting representative
$\mathscr K_{u,k}^{(q),[0],U}$ satisfies
\begin{equation}\label{eq:q-high-reproduction}
 \bigl\langle\mathscr K_{u,k}^{(q),[0],U},P
 \bigr\rangle_{\mathsf B_{-2u,k}^{E_U},y}
 =A_{-2u,k}^{(q),U}(P(x,U)).
\end{equation}

\item Let $M=-2u$ and $u+1\le k\le2u$, and put
$d=k-u$, $h=2u-k$, and
$\mathfrak R_{u,k}=r^d\mathcal P_h$.  Define
\[
 R_{u,k}^{U}(x)=S_{-2u,k}^{U}(x)\cap
  \bigl(\mathfrak R_{u,k}(x)\otimes E_U\bigr),
 \qquad
 R_{u,k}^{U}(y)=\jmath_{x\to y}R_{u,k}^{U}(x).
\]
Suppose, in addition, that
\begin{equation}\label{eq:q-radical-preservation}
 A_{-2u,k}^{(q),U}\bigl(R_{u,k}^{U}(x)\bigr)
 \subseteq\mathfrak R_{u,k}(x)\otimes E_U.
\end{equation}
For $z\in\{x,y\}$ put
\[
 \overline S_{u,k}^{U}(z)=S_{-2u,k}^{U}(z)/R_{u,k}^{U}(z),
 \qquad
 \overline V_{u,k}(z)=\mathcal P_k(z)/\mathfrak R_{u,k}(z).
\]
Then the projector induces the well-defined quotient map
\begin{equation}\label{eq:q-quotient-map-codomain}
 \overline A_{u,k}^{(q),U}:
 \overline S_{u,k}^{U}(x)
 \longrightarrow\overline V_{u,k}(x)\otimes E_U,
 \qquad
 [P]\longmapsto[A_{-2u,k}^{(q),U}P],
\end{equation}
and its restriction defines
\[
 A_{u,k}^{(q),\mathrm{rad},U}:
 R_{u,k}^{U}(x)\longrightarrow
 \mathfrak R_{u,k}(x)\otimes E_U.
\]
Applying \Cref{lem:finite-block-coevaluation} in the ambient quotient with
$\overline{\mathsf B}_{u,k}^{E_U}$ and in the ambient radical with
$\mathsf C_{u,k}^{E_U}$ gives tensor representatives
$\mathscr K_{u,k}^{(q),\mathrm{quo},U}$ and
$\mathscr K_{u,k}^{(q),\mathrm{rad},U}$ satisfying
\begin{align}
 \bigl\langle\mathscr K_{u,k}^{(q),\mathrm{quo},U},[P]
 \bigr\rangle_{\overline{\mathsf B}_{u,k}^{E_U},y}
 &=\overline A_{u,k}^{(q),U}[P],
 \label{eq:q-quotient-reproduction}\\
 \bigl\langle\mathscr K_{u,k}^{(q),\mathrm{rad},U},R
 \bigr\rangle_{\mathsf C_{u,k}^{E_U},y}
 &=A_{u,k}^{(q),\mathrm{rad},U}(R).
 \label{eq:q-radical-reproduction}
\end{align}
Without \eqref{eq:q-radical-preservation}, no quotient or radical
representative is asserted.
\end{enumerate}
In (i)--(iii), and in the radical statement of (iv), the represented output
belongs to the image of the right Fischer projector and is right
$q$-monogenic in $x$.  In the quotient statement of (iv), the output is the
class of such a $q$-monogenic representative.
\end{theorem}

\begin{proof}
The scalar forms used in (i)--(iii) are nondegenerate by the generic
reproduction theorem, the low-degree residue theorem, and the high-degree
part of \Cref{thm:confluent-supersphere-duality}, respectively.  Their
coefficient-dual extensions are nondegenerate by
\eqref{eq:coefficient-dual-pairing}, so
\Cref{lem:finite-block-coevaluation} applies with
$A=A_{M,k}^{(q),U}$.  In the collision range,
\eqref{eq:q-radical-preservation} is precisely the condition ensuring that
\eqref{eq:q-quotient-map-codomain} is independent of the chosen
representative and that the restricted map takes values in the ambient
radical.  The nondegenerate quotient and radical forms are those of
\eqref{eq:graded-confluent-pairing}.  The final handed-nullity statements
follow because every ambient representative on the right is obtained by the
Fischer projector.
\end{proof}

\begin{remark}\label{rem:collision-compatibility-open}
Condition \eqref{eq:q-radical-preservation} is not proved here for a general
nontrivial collision block.  It is retained as the exact compatibility
condition required by the quotient and radical constructions.  Determining
natural blocks for which it holds, or modifying the projector so that it is
automatic, remains open.
\end{remark}

Two further calculations delimit the present construction.  They show that
the natural scalar projected symbol is not itself a reproducing idempotent and
that the intrinsic derivative used here does not preserve the standard
full-frame quotient relation.

\begin{proposition}
\label{prop:q-obstructions}
The following obstructions rule out the most direct scalar and full-frame
versions in the indicated real parameter ranges.  For radial-algebra-valued
kernels define the
componentwise Pizzetti convolution
\[
 (K_1*K_2)(x,z)=\mathcal I_{M,y}^{\SSph}
 \bigl(K_1(x,y)K_2(y,z)\bigr).
\]
\begin{enumerate}[label=\textup{(\roman*)}]
\item For $M\ge3$ and $0<q<1$, the convolution of the degree-one
projected symbol with itself is proportional to that symbol only if
\begin{equation}\label{eq:degree-one-proportionality}
 \bigl([M-1]_q-(M-1)\bigr)\bigl([M-1]_q+1\bigr)=0,
\end{equation}
which never occurs in this range.
\item Let $M\ge2$, $0<q<1$, and let $e_1,\ldots,e_M$ be a complete
Clifford frame.  Let $\partial_{x,q}^{E,\mathrm R}$ denote the
Euclidean specialization of the intrinsic radial derivative.  For the
standard frame relation
$R_x=x+\frac12\sum_j\{x,e_j\}e_j$ one has
\begin{equation}\label{eq:frame-descent-defect}
 \partial_{x,q}^{E,\mathrm R}(R_x)=[M]_q-M\ne0.
\end{equation}
Thus this intrinsic radial derivative does not descend through the standard
full-frame quotient.
\end{enumerate}
\end{proposition}

\begin{proof}
Put $L=[M-1]_q$.  From \eqref{eq:initial-coeffs},
\[
 Z_{M,1}^{(q)}(x,y)=\frac{Ls_{xy}-2\omega_{xy}}{L+1}.
\]
The invariant second moments are
\begin{align*}
 \mathcal I_y^{\SSph}(s_{xy}s_{yz})&=-\frac{2\sigma_M}{M}s_{xz},\\
 \mathcal I_y^{\SSph}(s_{xy}\omega_{yz})
 =\mathcal I_y^{\SSph}(\omega_{xy}s_{yz})
 &=-\frac{2\sigma_M}{M}\omega_{xz},\\
 \mathcal I_y^{\SSph}(\omega_{xy}\omega_{yz})
 &=\frac{\sigma_M}{M}
 \left(-\frac{M-1}{2}s_{xz}+(2-M)\omega_{xz}\right),
\end{align*}
where $\sigma_M=2\pi^{M/2}/\Gamma(M/2)$.  These identities follow
by applying the quadratic Pizzetti moment and
polarizing in the external vectors; equivalently, they are the scalar and
bivector components of the invariant tensor
$\mathcal I_y(y\otimes y)$.  For the corresponding superspace second-moment
formula, see \cite[Section~4]{DeBieSommen2007Spherical}.  Substitution shows that
$\mathcal I_y(Z(x,y)Z(y,z))$ has the same scalar-to-bivector ratio as
$Z(x,z)$ if and only if
$(L-(M-1))(L+1)=0$.  For $M\ge3$ and $0<q<1$ one has
$0<L<M-1$, so neither factor vanishes.  This proves (i).

For (ii), the Euclidean one-vector rules give
$\partial_{x,q}^{E,\mathrm R}x=[M]_q$ and
\[
 \partial_{x,q}^{E,\mathrm R}\!\left(
 \tfrac12\{x,e_j\}e_j\right)=-1
\]
for each frame vector.  Summing over $j$ yields
$\partial_{x,q}^{E,\mathrm R}(R_x)=[M]_q-M$.  Since
$[M]_q<M$ for $M\ge2$ and $0<q<1$, the coordinate-kernel relation $R_x=0$
is not preserved by the intrinsic radial derivative.  Hence the derivative
cannot descend through the standard full-frame quotient.
\end{proof}

\section{Conclusion}

The meromorphic generating kernel
\[
 \mathscr G_\mu(\rho;x,y)
 =\frac{\Gamma(\mu/2)}{2\pi^{\mu/2}}
  \bigl(1+\rho\{x,y\}+\rho^2x^2y^2\bigr)^{-\mu/2}
\]
gives a uniform description of supersphere reproduction at generic and
exceptional superdimensions.  At $M=-2u$, its residue and logarithmic finite
part recover the duality structure of the exceptional Fischer filtration.
The residue controls the renormalized low-degree pairing and the transported
collision radical; the finite part controls the collision quotient and the
ordinary high-degree pairing.  The head--socle calculation shows, in
addition, that the generalized harmonic blocks retain a nondegenerate duality
despite their nonsemisimple structure.

The radial $q$-application provides a covariant two-point zonal null symbol
and identifies the obstruction to a naive one-sided transfer of scalar
Pizzetti reproduction.  The nondegenerate scalar pairings nevertheless yield
finite-block tensor representatives of the localized $q$-Fischer projection.
In the collision range, compatibility with the exceptional radical remains a
separate condition.  These results delimit the part of the deformation that
is intrinsic to the radial calculus without obscuring the main exceptional
supersphere theory.

The present analysis is algebraic and assumes a nonzero bosonic dimension.
The purely fermionic exceptional case is not covered.  Extending the
logarithmic kernels to Poisson, Cauchy, or boundary-value problems on the
superball would require analytic function spaces, boundary traces, and
convergence estimates, and is a natural direction for further work.

\section*{Statements and Declarations}

\medskip
\noindent\textbf{Funding.}\par
This work was co-funded by the University of Ostrava, Grant
No.~SGS05/P\allowbreak\v{R}F/2026. Yifan Zhang is also co-funded by the European Union under the REFRESH -- Research Excellence For REgion Sustainability and High-tech Industries project, No.~CZ.10.03.01/00/22003/0000048, via the Operational Programme Just Transition.

\medskip
\noindent\textbf{Competing interests.}\par
The authors have no relevant financial or non-financial interests to
disclose.

\medskip
\noindent\textbf{Data availability.}\par
No datasets were generated or analysed during the current study.

\medskip
\noindent\textbf{Author contributions.}\par
All authors contributed to the conception and development of the results,
verification of the proofs, and preparation of the manuscript.  All authors
read and approved the final manuscript.

\bibliographystyle{plainurl}
\bibliography{references}

\end{document}